\documentclass[10pt]{article}
\usepackage{epsfig,verbatim}
\usepackage{amsmath}
\usepackage{appendix}

\usepackage{amsmath, amssymb, bbm, xspace}
\usepackage{epsfig}
\usepackage{longtable}
\usepackage{color}
\usepackage{mathrsfs}
\usepackage{subfigure}

\usepackage{courier}

\newtheorem{theorem}{Theorem}[section]

\newtheorem{lemma}{Lemma}[section]

\newtheorem{proposition}[theorem]{Proposition}

\newtheorem{corollary}[theorem]{Corollary}

\newtheorem{definition}{Definition}[section]
\def\bkE{{\rm I\kern-.17em E}}
\def\bk1{{\rm 1\kern-.17em l}}
\def\bkD{{\rm I\kern-.17em D}}
\def\bkR{{\rm I\kern-.17em R}}
\def\bkP{{\rm I\kern-.17em P}}
\def\Cbb{{\mathbb{C}}}

\def\fnat{{\bf F}^{\textrm{nat}}}

\def\bkZ{{\bf{Z}}}

\def\bkE{{\rm I\kern-.17em E}}
\def\bk1{{\rm 1\kern-.17em l}}
\def\bkD{{\rm I\kern-.17em D}}
\def\bkR{{\rm I\kern-.17em R}}
\def\bkP{{\rm I\kern-.17em P}}

\newcommand{\scg}{shared-constraint game\@\xspace} 
\newcommand{\scgs}{shared-constraint games\@\xspace}

\def\bkZ{{\bf{Z}}}
\def\b12{(\beta_1,\beta_2)}
\def\Epec{{\mathscr{E}}}

\newenvironment{proof}[1][]{{\noindent \bf Proof #1: }}{\hfill \qed \vspace{6pt}\\ }

\newcounter{example}
\renewcommand{\theexample}{\thesection.\arabic{example}}
\newenvironment{examplec}[1][]{\refstepcounter{example}
\par\medskip \noindent%
   \textbf{Example~\theexample. #1} \rmfamily}{\hfill$\square$\vspace{6pt}}

\newcounter{remark}
\renewcommand{\theremark}{\thesection.\arabic{remark}}
\newenvironment{remarkc}[1][]{\refstepcounter{remark}
\par\medskip \noindent%
   \textbf{Remark~\theremark. #1} \rmfamily}{\hfill $\square$\vspace{6pt}}

\def\t{^\top}
\def\th{^{\rm th}}

\newlength{\noteWidth}
\long\def\notes#1{\ifinner
{\tiny #1}
\else
\marginpar{\parbox[t]{\noteWidth}{\raggedright\tiny #1}}
\fi\typeout{#1}}

 \def\notes#1{\typeout{read notes: #1}} 

\def\mi{^{-i}}

\newcommand{\I}[1]{\mathbb{I}_{\{#1\}}}

\newcommand{\ie}{i.e.\@\xspace} 
\newcommand{\eg}{e.g.\@\xspace} 
\newcommand{\etal}{et al.\@\xspace} 

\newcommand{\Real}{\ensuremath{\mathbbm{R}}}

\newcommand{\dom}{\mathrm{dom}}
\newcommand{\minimize}[1]{\displaystyle\minim_{#1}}
\newcommand{\minim}{\mathop{\hbox{\rm minimize}}}

\DeclareMathOperator*{\st}{subject\;to}

\def\subject{\hbox{\rm subject to}}

\def\fnat{{\bf F}^{\rm nat}}

\def\Cbb{\mathbb{C}}

\def\ind{{\rm ind}}

\def\range{{\rm range}}

\def\half  {{\textstyle{1\over 2}}}

\def\spose#1{\hbox to 0pt{#1\hss}}

\def\text #1{\hbox{\quad#1\quad}}

\def\nthinsp{\mskip -2   mu}

\def\S{_{\scriptscriptstyle S}}

\def\superstar{^{\raise 0.5pt\hbox{$\nthinsp *$}}}
\def\SUPERSTAR{^{\raise 0.5pt\hbox{$*$}}}

\def\lamstarT {\lambda^{\raise 0.5pt\hbox{$\nthinsp *$}T}}

\def\lambar{\bar\lambda}

\def\lambdabar{\lambar}

\def\mubar{\skew3\bar \mu}

\def\Ascr{{\cal A}}
\def\Bscr{{\cal B}}
\def\Fscr{{\cal F}}

\def\Iscr{{\cal I}}

\def\Lscr{{\cal L}}

\def\Tscr{{\cal T}}

\def\Sscr{{\cal S}}

\def\Nscr{{\cal N}}
\def\Rscr{{\cal R}}
\def\Gscr{{\cal G}}

\def\aur{\;\textrm{and}\;}

\def\eef{\;\textrm{if}\;}

\def\non{\nonumber}
\def\SOL{{\rm SOL}}
\def\VI{{\rm VI}}

\let\forallnew\forall
\renewcommand{\forall}{\forallnew\ }
\let\forall\forallnew
\newcommand{\mlmfg}{multi-leader multi-follower game\@\xspace}
\newcommand{\mlmfgs}{multi-leader multi-follower games\@\xspace}
\def\ds{\displaystyle}
		\def\bkE{{\rm I\kern-.17em E}}
		\def\bk1{{\rm 1\kern-.17em l}}
		\def\bkD{{\rm I\kern-.17em D}}
		\def\bkR{{\rm I\kern-.17em R}}
		\def\bkP{{\rm I\kern-.17em P}}
		\def\bkY{{\bf \kern-.17em Y}}
		\def\bkZ{{\bf \kern-.17em Z}}
		\def\bkC{{\bf  \kern-.17em C}}

{\begin{list}{}%
         {\setlength{\leftmargin}{#1}}%
         \item[]%
}
{\end{list}}

		\def\bsp{\begin{split}}
		\def\beq{\begin{eqnarray}}
		\def\bal{\begin{align*}}
		\def\bc{\begin{center}}
		\def\be{\begin{enumerate}}
		\def\bi{\begin{itemize}}
		\def\bs{\begin{small}}
		\def\bS{\begin{slide}}
		\def\ec{\end{center}}
		\def\ee{\end{enumerate}}
		\def\ei{\end{itemize}}
		\def\es{\end{small}}
		\def\eS{\end{slide}}
		\def\eeq{\end{eqnarray}}
		\def\eal{\end{align*}}
		\def\esp{\end{split}}
		\def\qed{ \vrule height7.5pt width7.5pt depth0pt}  

		\def\problemsmall#1#2#3#4{\fbox
		 {\begin{tabular*}{0.48\textwidth}
			{@{}l@{\extracolsep{\fill}}l@{\extracolsep{6pt}}l@{\extracolsep{\fill}}c@{}}
				#1 & $\minimize{#2}$ & $#3$ & $ $ \\[5pt]
					 & $\subject\ $    & $#4$ & $ $
			\end{tabular*}}
			}

		\def\problem#1#2#3#4{\fbox
		 {\begin{tabular*}{0.80\textwidth}
			{@{}l@{\extracolsep{\fill}}l@{\extracolsep{6pt}}l@{\extracolsep{\fill}}c@{}}
				#1 & $\minimize{#2}$ & $#3$ & $ $ \\[5pt]
					 & $\subject\ $    & $#4$ & $ $
			\end{tabular*}}
			}

	\def\cp2problem#1#2#3#4{\fbox
		 {\begin{tabular*}{0.9\textwidth}
			{@{}l@{\extracolsep{\fill}}l@{\extracolsep{6pt}}l@{\extracolsep{\fill}}c@{}}
				#1 & & $#4 $ 
			\end{tabular*}}}

\newcommand{\pmat}[1]{\begin{pmatrix} #1 \end{pmatrix}}
		
		\renewcommand{\emph}[1]{\textbf{#1}}

		\def\bkE{{\rm I\kern-.17em E}}
		\def\bk1{{\rm 1\kern-.17em l}}
		\def\bkD{{\rm I\kern-.17em D}}
		\def\bkR{{\rm I\kern-.17em R}}
		\def\bkP{{\rm I\kern-.17em P}}
		
		\def\bkZ{{\bf{Z}}}

\newcommand {\beeq}[1]{\begin{equation}\label{#1}}
\newcommand {\eeeq}{\end{equation}}
\newcommand {\bea}{\begin{eqnarray}}
\newcommand {\eea}{\end{eqnarray}}

\def\texitem#1{\par\smallskip\noindent\hangindent 25pt
               \hbox to 25pt {\hss #1 ~}\ignorespaces}

\def\st{\mbox{subject to}}

\usepackage{epsfig,amsmath,amssymb}
\usepackage{longtable}
\usepackage{color}
\usepackage{mathrsfs}

\def\texitem#1{\par\smallskip\noindent\hangindent 25pt
               \hbox to 25pt {\hss #1 ~}\ignorespaces}

\def\st{\mbox{subject to}}

\def\PS{{\rm P}^{\rm ae}}
\def\OmegaS{\Omega^{\rm ae}}
\def\LiS{{\rm L}^{\rm ae}_i}
\def\UpsilonS{\Upsilon^{\rm ae}}
\def\EpecS{\Epec^{\rm ae}}
\def\FscrS{\Fscr^{\rm ae}}
\def\Iscr{\mathcal I}
\def\bl{\rm bl}

\renewcommand{\emph}[1]{\textbf{#1}}
\def\bkE{{\rm I\kern-.17em E}}
\def\bk1{{\rm 1\kern-.17em l}}
\def\bkD{{\rm I\kern-.17em D}}
\def\bkR{{\rm I\kern-.17em R}}
\def\bkP{{\rm I\kern-.17em P}}

\def\bkZ{{\bf{Z}}}
\def\b12{(\beta_1,\beta_2)}
\def\Epec{{\mathscr{E}}}
\newcommand{\gap}{\vspace{0in}}
\def\us#1{{{\color{red}#1}}}
\def\ae{\rm ae}
\def\cc{\rm cc}
\def\bl{\rm bl}

\begin{document}
\title{\bf A Shared-Constraint Approach to Multi-leader Multi-follower Games}
\author{Ankur A. Kulkarni \and Uday V. Shanbhag\thanks{The first author
is with the Systems and Control Engineering group at the Indian Insitute
	of Technology Bombay, Mumbai 400076, India, while the second author
	is with the Department of Industrial and Manufacturing Engineering
	at the Pennsylvania State University at University Park.  They  are
	reachable at ({\tt kulkarni.ankur@iitb.ac.in,udaybag@psu.edu}).  The
	work of the second author has been partially funded by the
	NSF CMMI 124688 (CAREER). Finally, the authors would like to thank
	Profs.  T.~Ba\c{s}ar and J.-S.\ Pang for their suggestions and
	comments.}}
\date{}
\maketitle
\begin{abstract}
Multi-leader multi-follower games are a class of hierarchical games in
which a collection of leaders compete in a Nash game constrained by the
equilibrium conditions of another  Nash game amongst the followers.
{The resulting \textit{equilibrium problem with equilibrium constraints} is complicated by nonconvex
agent problems and therefore providing tractable conditions for
existence of global or even local equilibria for it has proved
challenging.} Consequently, much of the extant research on this
topic is either model specific or relies on weaker notions of
equilibria.  We consider a \textit{modified} formulation in which every leader is
cognizant of the equilibrium constraints of all leaders. Equilibria of this
modified game {\em contain} the equilibria, if any, of the
original game. 
The new formulation has a constraint structure called {\em shared constraints},
and our main result shows that if the leader objectives admit a potential function,
the global minimizers of the potential function over the
shared constraint  are equilibria of the modified formulation.  
We provide another existence result using fixed point theory that does not require potentiality.  
Additionally, local minima, B-stationary, and
strong-stationary  points of this minimization
are shown to be local Nash equilibria, Nash
B-stationary, and Nash strong-stationary points of the
corresponding \mlmfg. We demonstrate the relationship between {\em variational equilibria}
associated with this modified shared-constraint game and
equilibria of the
original game from the standpoint of the multiplier sets and show how equilibria of the original formulation may be recovered. We note
through several examples that such {potential} multi-leader
multi-follower games capture a breadth of application problems of
interest and demonstrate our findings on a multi-leader multi-follower
Cournot game.
\end{abstract}

\section{Introduction} \label{sec:intro}
This paper concerns  \mlmfgs where multiple Stackelberg leaders
participate in a simultaneous move game, multiple followers
particpate in a subsequent simultaneous move game taking the
strategies of the leaders as given and leaders make decisions subject to
the equilibrium conditions arising from the game between followers. This follower-level equilibrium need not be unique as a function of the leader strategy profile and leaders and followers may have a continuum of
strategies, whereby an equilibrium of the game between leaders is
characterized by an analytically difficult problem. This problem is popularly referred to as an {\bf e}quilibrium {\bf
	p}rogram with {\bf e}quilibrium {\bf c}onstraints (EPEC). We are concerned
with the central question of the existence of an equilibrium to this problem.

Games described above arise organically in the modeling of a sequence of
clearings such as the day-ahead and real-time clearings in power
markets. Increases in the computing capability have enabled ``rational''
firms to make decisions with longer time horizons and by taking into
account explicitly the situation that would emerge in later clearings.
Consequently, models of these strategic interactions require a firm to
be not just strategic with respect to other firms but also cognizant
of the real-time market clearing to
follow~\cite{yao08modeling,su07analysis,SIGlynn11}.  Technically such
firms must be modeled as \textit{leaders} that participate in a Nash
game subject to the equilibrium amongst another set of participants
called \textit{followers}. The resulting game  is a \mlmfg.

While such models are indeed reasonable representations of the
hierarchical competitive structure, general results on existence
of equilibria are scarce. Definitive statements on the existence of
equilibria have been obtained mainly for \mlmfgs with specific
structure~\cite{sherali84multiple,sherali83stackelberg,su07analysis} and
for models arising from specific
applications~\cite{demiguel09stochastic,murphy10forward}.  Furthermore,
these results are reliant on the follower-level problem having a clean
structure, in most cases uniqueness of its equilibrium as a function of
the leader strategies, so that upon substituting this equilibrium into
the leader's objectives the resulting {\em implicit problem} has a form
amenable to analysis via standard fixed-point theorems\footnote{A few
	other lines of work have shown the solvability of stationarity
		conditions of the problems of the leaders~\cite{hu07epec,
			SIGlynn11,outrata04note,pang10local,leyffer05multileader,su05equilibrium}.}.
To the contrary, the main results of this paper impose no such requirements and probe EPECs from an entirely new perspective. Our contributions are as follows.
\begin{enumerate}
\item[(i)] {We present
a {\em modified} formulation of multi-leader multi-follower competition
in which there exists a {\em common (or shared) constraint} that constrains each
player's optimization problem~\cite{rosen65existence,facchinei07ongeneralized,kulkarni09refinement}.  
The \textit{conventional} formulation (which has been analyzed by the above surveyed results) of a \mlmfg bears a close resemblance to
shared-constraint game, 
	but it	is technically not a \scg, thereby motivating the need for a modified formulation. 
We show that if the  leader objectives admit a potential function, then
any minimizer of the 
potential function over a shared constraint  is an equilibrium of the
modified game. Furthermore, equilibria of the conventional formulation are equilibria of the modified formulation. {Additionally, we show that local
minimizers, B-stationary points, and strong-stationary points of this potential function over the shared constraint are local Nash
equilibria, Nash B-stationary points, and Nash strong-stationary points of the modified game.} } We  further show how the structure of shared constraints can be exploited 
in games that do not admit potential functions via advanced fixed-point
theorems.
\item [(ii)] We present a clear
understanding of the relationship between equilibria associated with the
two formulations. First, it can be seen that modified game is a
shared-constraint game that admits at least two sets of generalized Nash
equilibria of interest: (i) Equilibria of
the original game; and (ii) Equilibria characterized by a ``common'' or
consistent Lagrange multiplier that can be viewed as \textit{variational
equilibria}~\cite{kulkarni09refinement}  for which existence statements are available. 
\end{enumerate}
	At a high level, this paper is motivated by the view that the
	competition between multiple Stackelberg leaders is not an obscure
	or pathological setting and may thereby admit a mathematical model
	that allows for a reasonably general existence theory.  The
	shared-constraint model is an attempt in this direction. Our
	modified model can be viewed as either an alternative model or it can be
	seen as a	vehicle for  developing existence statements for the
	conventional model. 



The remainder of the paper is organized into five sections.  In
Section~\ref{sec:conventional}, we introduce the conventional formulation and survey \mlmfgs studied in practice
provide some background.  In Section \ref{sec:shared} we present
the modification that leads to a shared-constraint game and present
existence results for it. Recovery of equilibria of the original
formulation is examined in Section~\ref{sec:4}.  We apply our techniques towards the analysis
of a hierarchical Cournot game in Section~\ref{sec:cour} and 
conclude in Section~\ref{conc} with a brief summary.

\section{Multi-leader multi-follower games: examples and background} \label{sec:conventional}
{This section begins with a general formulation for such
	games and the associated equilibrium problem in
	Section~\ref{sec:form}. } In Section~\ref{sec:examples}, we discuss
	several examples considered in literature with the intent of noting
	that in a majority of these instances, the associated objective
	functions of the leaders admit a potential function, thereby also
	noting the utility of this class in practice. Section 2.3 contains a few preliminaries for the results to follow.

\subsection{Conventional formulation {of \mlmfgs}}\label{sec:form}
Let $\Nscr = \{1,2,\hdots,N\}$ denote the set of leaders. In the
conventional formulation of \mlmfgs, leader $i \in \Nscr$ solves a parametrized
optimization problem of the
following kind
$$
	\problem{L$_i(x^{-i},y^{-i})$}
	{x_i,y_i}
	{\varphi_i(x_i,y_i;x^{-i})}
				 {\begin{array}{r@{\ }c@{\ }l}
		x_i &\in& X_i, \\
		y_i & \in& Y_i, \\
		y_i &\in&  \mbox{SOL}(G(x_i,x^{-i},\cdot),K(x_i,x^{-i})),
	\end{array}}
	$$
where, $$x^{-i} \triangleq
	(x_1,\hdots,x_{i-1},x_{i+1},\hdots,x_N) \quad \mbox{and} \quad
	(\bar{x}_i,x^{-i}) \triangleq
	(x_1,\hdots,x_{i-1},\bar{x}_i,x_{i+1},\hdots,x_N).$$  
In this formulation, the leader makes two decisions: his strategy denoted $x_i \in \Real^{m_i}$, 
and his \textit{conjecture} about the equilibrium of the followers, denoted $ y_i $. The choice of the optimal strategy $ x_i $ can vary with the precise follower equilibrium that occurs; if this equilibrium is not unique, the leader must make a conjecture about which of the several follower equilibria will actually emerge. Since this conjecture $ y_i $ influences the choice of $ x_i $, $ y_i $ is also a \textit{decision}. 
The formulation above corresponds to an \textit{optimistic} one since the leader picks the $y_i$ that is most favorable to him~\cite{luo96mathematical}. The \textit{pessimistic} formulation (more standard amongst control theorists~\cite{basar99dynamic}) involves a `$ \min_{x_i}\max_{y_i} $'.

The set of equilibria of the game between followers are given as $\SOL(G(x_i,x^{-i},\cdot),K(x_i,x^{-i}))$, which stands for the solution set of the variational inequality (VI), VI$(G(x,\cdot),K(x))$,  
	parametrized by 	the tuple of leader strategies  $x=(x_1,\hdots,x_{N})$.  For each $x$, we let
\begin{equation}
\Sscr(x) \triangleq \mbox{SOL}(G(x_i,x^{-i},\cdot),K(x_i,x^{-i})). \label{eq:sscr}
\end{equation} 
Throughout we assume that $ K $ is continuous as a set-valued map and $ G $ is a continuous mapping of all variables.

The sets $X_i$ and $Y_i$ are assumed to
be closed convex sets.  For each $i$, objective function $\varphi_i: X \times Y \rightarrow \Real$, where $X
\triangleq \prod_{i=1}^N X_i$ and $Y \triangleq \prod_{i=1}^N Y_i$, is 
assumed to be continuous.  Let $y =
(y_1,\hdots,y_N)$ and   $\Omega_i(x^{-i},y^{-i})$  be the feasible
region of L$_i(x^{-i},y^{-i})$, {given by}
\begin{align}\label{omegai}
	\Omega_i(x^{-i},y^{-i}) \triangleq 
	\left\{(x_i,y_i) \in \Real^{n_i} \ \left\lvert \ \begin{array}{clc} x_i & \in X_i, \\
y_i &\in Y_i,\\ 
y_i &\in \Sscr(x)
\end{array} \right. \right\}, \end{align}
where $\Real^n$ is the ambient space of the tuple $(x_i,y_i)$. Notice that, $ \Omega_i(x\mi,y\mi) $ is in fact independent of $ y\mi.$ However we use this notation to maintain consistency with other notation we introduce in the context of shared constraints.
Let $\Omega(x,y)$ denote the Cartesian product of $\Omega_i(x\mi,y\mi):$
\begin{equation} \Omega(x,y) \triangleq \prod_{i=1}^N
\Omega_i(x^{-i},y^{-i}).
\label{eq:omega} \end{equation}
An important object in our analysis is the set $\Fscr$ defined as \begin{align}
\Fscr \triangleq \left\{(x,y) \in \Real^n \ \left\lvert \ \begin{array}{clc} x_i & \in X_i, \\
y_i &\in Y_i,\\ 
y_i &\in \Sscr(x), & \quad i = 1,\hdots,N
\end{array} \right. \right\}. \label{eq:fscr} \end{align}
Clearly, $ \Fscr = \{(x,y) \in \Real^n: (x,y) \in \Omega(x,y)\}. $
We refer to $\Omega$ as the {\em feasible region mapping} and denote
this multi-leader multi-follower game or EPEC by $\Epec$. 
\begin{definition}[Global Nash equilibrium]
	Consider the multi-leader multi-follower game $\Epec$. The {\em global
	Nash equilibrium}, or {\em equilibrium}, of $\Epec$ is a point $(x,y) \in \Fscr$ that satisfies the following: 
\begin{align}\label{br}
\varphi_i(x_i,y_i;x^{-i}) &\leq \varphi_i(u_i,v_i;x^{-i})  \quad \forall (u_i,v_i) \in \Omega_i(x^{-i},y^{-i}), \quad i =1,\hdots,N.
\end{align}
\end{definition}
Local notions of equilibria will be defined later in the paper.

\subsection{Examples {of multi-leader multi-follower games}}\label{sec:examples}
The multi-leader multi-follower game is inspired by a strategic game in
economic theory referred to as a Stackelberg
game~\cite{vonstack52theory}. In such a game, the leader is aware of the strategic consequences of the 
follower's reaction and employs that knowledge in making a first move.
The follower observes this move and responds as per its optimization
problem. An  extension to this regime was provided by Sherali et
al.~\cite{sherali83stackelberg} where a set of followers compete in a Cournot
game while a leader makes a decision constrained by the equilibrium of
this game. While multi-leader generalizations were touched upon by
Okuguchi~\cite{okuguchi76expectations}, Sherali~\cite{sherali84multiple}
presented amongst the first models for multi-leader multi-follower
games in a Cournot regime. 
{A majority of multi-leader multi-follower  game-theoretic models
	appear to fall into three broad categories. We provide a short
	description of the games arising in each category}:
\paragraph{Hierarchical Cournot games:} In a hierarchical Cournot game, leaders compete in
a Cournot game and are constrained by the reactions of a set of
followers that also compete in a Cournot game. {We discuss a setting
comprising of $N$ leaders and $M$ followers, akin to that
proposed by Sherali~\cite{sherali84multiple}. Suppose the $i\th$ leader's
decision is denoted by $x_i$ and the follower
strategies conjectured by leader $i$ are collectively denoted by
$\{y^f_i\}_{f=1}^M$ where $f$ denotes the follower index.  
Given the leaders' decisions, follower $f$ participates in a Cournot
game in which it solves the following parametrized problem}:
$$ \problem{F$(\bar{y}^{-f},x)$}
	{y^f}
	{\half c_f (y^f)^2 - y^f p \left( \bar{y} + \bar{x}) \right) }
				 {\begin{array}{r@{\ }c@{\ }l}
y_f & \geq & 0,	
\end{array}}$$
where $p(.)$ denotes the price function associated with the follower Cournot
game, $\half c_f (y^f)^2$ denotes firm $f$'s quadratic cost of
production,  $\bar{x} \triangleq
\sum_i x_i$, $\bar{y} \triangleq \sum_{f} y^f$, and $\bar{y}^{-f}
\triangleq \sum_{j \neq f} y^j.$ Leader $i$ solves {the following parametrized problem:} 
$$	\problem{L$_i(x\mi,y\mi)$}
	{x_i,y_i}
	{\half d_i x_i^2 - x_ip\left(\bar{x}+\bar{y_i})\right)   }
				 {\begin{array}{r@{\ }c@{\ }l}
	y^f_i & = & \mbox{SOL(F}(\bar{y}^{-f}_i,x_i,x\mi)),\quad \forall\ f,\\
	x_i &\geq& 0,
\end{array}} $$   
{where $y_i^f \in \Real$ is leader $i$'s conjecture of follower $f$'s
equilibrium strategy, $y_i \triangleq \{y_i^f\}_{f=1}^M$, $\half d_i
x_i^2$ denotes the cost of production of leader $i$, $x^{-i}
\triangleq \{x_j\}_{j \neq i}$ and $y^{-i}
\triangleq \{y^f_j\}_{j \neq i,f=1}^M.$ The equilibrium of the resulting multi-leader
multi-follower is given by $\{(x_i,y_i)\}_{i=1}^N$ where
$(x_i,y_i)$ is a solution of  L$_i(x\mi,y\mi)$ for $i = 1, \hdots, N$.
In this regime, under identical leader costs,
Sherali~\cite{sherali84multiple} proved the existence and uniqueness of
the associated  equilibrium.  More recently, DeMiguel and
Xu~\cite{demiguel09stochastic} extended this result to stochastic
regimes wherein the price function is uncertain and the leaders solve
expected-value problems.}


\paragraph{Spot-forward markets:} {Motivated by the need to investigate
the role of forward transactions in power markets, there has been much
interest in strategic models where firms compete in the forward market
subject to equilibrium in the real-time market.}  Allaz and
Vila~\cite{allaz93cournot} examined a forward market comprising of two
identical Cournot firms and demonstrated that global equilibria exist in
such markets. 
Su~\cite{su07analysis} extended these existence
statements to a multi-player regime where firms need not have identical
costs. In such an $N$-player setting, given the forward decisions of the
players $\{x_i\}_{i=1}^N$, firm $i$ solves the following parametrized
problem in spot-market:
$$ \problem{S$(z\mi,x)$}
	{z_i}
	{c_iz_i - p \left(\bar{z} \right)(z_i-x_i) }
				 {\begin{array}{r@{\ }c@{\ }l}
z_i & \geq & 0,	
\end{array}}$$
$ y_{i,j} $ is the 
where 
$c_iz_i$ is the linear cost of producing $z_i$ units in the
spot-market, $ \bar{z} = \sum_{j}z_{j} $ and $p(.)$ is the price function in the spot-market. In the
forward market, firm $i$'s objective is given by its overall profit,
		which is given by $-p^f x_i - p(\bar{y}_i) (y_{i,i} - x_i) + c_iy_{i,i}$, where $p^f$ denotes the price in the forward market
		$ y_{i,j} $ is the anticipated equilibrium production by leader $ j $ and $ \bar{y}_i= \sum_jy_{i,j} $. By imposing the
no-arbitrage constraint that requires that $p^f = p(\bar{y}_i)$, the
forward market objective reduces to $c_i y_{i,i} -p(\bar{y}_i)y_{i,i}.$ Firm $i$'s
problem in the forward market is given by the following:
$$ \problem{L$(x^{-i})$}
	{x_i,y_i}
	{c_i y_{i,i} - p(\bar{y}_i)y_{i,i}}
				 {\begin{array}{r@{\ }c@{\ }l}
	y_{i,j} & \in & \mbox{SOL(S}(\bar{y}^{-j}_i,x_i,x\mi)),\quad \forall\ j.
\end{array}}$$
Note that while the spot-forward market problem is closely related to the hierarchical
Cournot game, it has two key distinctions. First, leader $i$'s cost is
a function of forward and spot decisions. Second, every leader's revenue
includes the revenue from the second-level spot-market sales. As a consequence, the
problem cannot be reduced to the hierarchical Cournot game, as observed
by Su~\cite{su07analysis}. In related work, Shanbhag, Infanger and
Glynn~\cite{SIGlynn11} conclude the existence of local equilibria in a
regime where each firm  employs a conjecture of the forward price
function.  Finally, in a constrained variant of the spot-forward game
examined by Allaz and Vila, Murphy and Smeers~\cite{murphy10forward}
prove the existence of global equilibria when firm capacities are
endogenously determined by trading on a capacity market and further
discover that Allaz and Vila's conclusions regarding the benefits of
forward markets may not necessarily hold. In electricity markets, there 
has been work beyond the papers mentioned above, in particular, by 
Henrion, Outrata, and Surowiec~\cite{henrion2009analysis} and Escobar and Jofr\'{e}~\cite{escobar2008equilibrium}.

{We conclude this section with two observations. First, almost all of the
existence results are model-specific and are not more generally
applicable to the class of multi-leader multi-follower games.  Second,
in all of the instances surveyed above, the leader objectives admit
a potential function. For instance, in
hierarchical Cournot games, if the associated price functions are
affine, then the resulting game is a potential \mlmfg
(cf.~\cite{monderer96potential}). In the spot-forward games, the
leader's objectives are dependent only on follower decisions;
consequently, the payoffs are independent of competitive decisions and
this can be immediately seen to be a potential \mlmfg.  } 


\subsection{Preliminaries}
This paper makes extensive use of \textit{shared constraints} and \textit{potential functions}. In this section we review these concepts.
\subsubsection{Background on shared-constraint games} \label{sec:formulations}
\label{sec:coupled}
 Shared-constraint games were introduced by Rosen \cite{rosen65existence}
as a generalization of the classical Nash game.  In a \scg, there exists
a set $\Cbb$ in the product space of strategies such that for any player $i$, and for any
tuple of strategies of other players (denoted $z\mi$), the feasible
strategies $z_i$ for player $i$ are those  that satisfy $(z_i,z\mi) \in
\Cbb$. In an $N$-person shared-constraint Nash game with player payoffs 
denoted by $\{f_1, \hdots, f_N\}$, player $i$ solves:
$$	\problemsmall{A$_i(z\mi)$}
	{z_i}
	{f_i(z_i;z\mi) }
				 {\begin{array}{r@{\ }c@{\ }l}
	(z_i,z\mi) & \in & \Cbb
\end{array}} $$

An equilibrium  $z=
(z_1,\hdots,z_N)$  satisfies the following:
\begin{equation}
z \in \Cbb, \quad f_i(z_1,\hdots,z_N) \leq f_i(z_1,\hdots,\bar{z}_i,\hdots,z_N) \qquad \forall \ \bar{z}_i \ \mbox{s.t.} \ (z_1,\hdots, \bar{z}_i,\hdots, z_N) \in \Cbb, \quad \forall\ i\in \Nscr.\label{eq:rosen} 
\end{equation}
 Equivalently, $z$ is an equilibrium if $z \in
\Omega^\Cbb(z)$ and for all $i$
$$f_i(z_1,\hdots,z_N) \leq f_i(z_1,\hdots,\bar{z}_i,\hdots,z_N) \qquad \forall \ \bar{z}_i \in \Omega_i^\Cbb(z^{-i}),$$ where 
$$\Omega^\Cbb(z) \triangleq \prod_{i=1}^N \Omega_i^\Cbb(z^{-i}) \quad
\aur \quad \Omega_i^\Cbb(z^{-i}) \triangleq \left\{ \bar{z}_i \ | \
		(\bar{z}_i;z^{-i}) \in \Cbb \right\}. $$
It is easy to show~\cite{kulkarni09refinement} that $ z \in \Cbb \iff z \in \Omega^\Cbb(z) $		

		The feasible region mapping $\Omega$ defined in \eqref{eq:omega} (where
				$\Omega_i(x^{-i},y^{-i})$ is the feasible region of L$_i(x^{-i},y^{-i})$) 
		is a shared constraint if $\Omega$ has the following structure: for $(x,y)$ 
		in the domain of $\Omega$,  
		\begin{equation}(u,v) \in \Omega(x,y) 
		\iff (u_i,x^{-i},v_i,y^{-i}) \in \Fscr \quad \forall \ i \in \Nscr. \label{eq:shared}
		\end{equation} 
		(Recall that $\Fscr$ was defined in \eqref{eq:fscr} and is the set of fixed points of $\Omega$).   
		It is easy to check that this condition does not hold in general for the
		mapping $\Omega$, whereby $\Epec$ is in general not a shared constraint game.

Instead $ \Epec $ is a {\em coupled constraint game} or \textit{abstract economy}~\cite{arrow54existence} where constraints of a player are dependent on the choices of other players, but it does not obey the form of \eqref{eq:shared}. In such a game,
	an equilibrium is a point $z$ such that 
$$z \in  \prod_{i=1}^N \Omega^{\rm NS}_i(z^{-i}), \qquad f_i(z_1,\hdots,z_N) \leq f_i(z_1,\hdots,\bar{z}_i,\hdots,z_N) \qquad \forall \ \bar{z}_i \in \Omega^{\rm NS}_i(z^{-i}), \quad \forall  \ i \in \Nscr.$$ 
Here $\Omega^{\rm NS}_i$ is {\em any} set-valued map, not
necessarily of the form  of a shared constraint. 
The key difference between $\Omega^{\rm NS}$ and $\Omega^{\Cbb}$ is that 
$\Omega^\Cbb$ is completely defined by its fixed point set ($\Cbb$), whereas $\Omega^{\rm NS}$ is not.
However in both cases, the equilibrium is a point that lies in the fixed
point set (given by $\cap_{i=1}^N \Cbb_i$ for $\Omega^{\rm NS}$, 
where $\Cbb_i$ is the graph of $\Omega_i^{\rm NS}$). The shared-constraint game is a special case of this with $\Cbb_i=\Cbb_j=\Cbb= \cap_k \Cbb_k$ for all $i,j,k$. 

\gap


\gap

Shared constraint games 
arise naturally when players face a {common constraint}, \eg in a bandwidth sharing game, 
and are an area of flourishing recent research; see \cite{facchinei07generalized,kulkarni09refinement}. 
Less is known in literature about coupled constraint games  without shared 
constraint even with convex constraints. 
On the contrary, much has been said about 
shared constraint games when the common constraint $\Cbb$ is convex (see particularly, 
the works of Rosen \cite{rosen65existence}, Facchinei \etal \cite{facchinei07ongeneralized}, Kulkarni and Shanbhag  
\cite{kulkarni09refinement,kulkarni12revisiting,kulkarni09cdcrefinement} and 
Facchinei and Pang \cite{pang09convex}).


\subsubsection{Potential games} 
Potential games were introduced by Monderer and Shapley~\cite{monderer96potential}. In the context of $ \Epec, $ we say
 \begin{definition}[Potential game]\label{def:pot}
 A multi-leader multi-follower game $ \Epec $ where leaders have objective
 functions $\varphi_i, i \in \Nscr$ is a potential game if  there exists a function $\pi$, called potential function,
 	 such that for all $i\in \Nscr$, for all $(x_i,x\mi) \in X,
 	 (y_i,y\mi) \in Y$ and for all $x_i' \in X_i,y_i' \in Y_i$
 \begin{align}
 \varphi_i(x_i,y_i;x\mi,y\mi)-\varphi_i(x_i',y_i';x\mi,y\mi) = \pi(x_i,y_i;x\mi,y\mi) - \pi(x_i',y_i';x\mi,y\mi). \label{eq:pot}
 \end{align}
 \end{definition}
  If $\varphi_i$ is a continuously differentiable function for $ i = 1,
  \hdots, N$, then it follows~\cite{monderer96potential} that $\pi$ is
  continuously differentiable. In this case $\pi$ is a potential function
  if and only if \begin{equation}
 \nabla_i \varphi_i(x_i,y_i;x\mi,y\mi) = \nabla_i \pi(x_i,y_i;x\mi,y\mi) \quad \quad \forall\ x,y, \forall\ i, \label{eq:potdiff}
 \end{equation}
 where $\nabla_i = \frac{\partial}{\partial (x_i,y_i)}$. \ie, if and only if 
 the mapping 
 \begin{equation}
 F \triangleq (\nabla_1\varphi_1,\hdots,\nabla_N\varphi_N) \label{eq:fdef}
 \end{equation}
   is integrable. The following lemma follows from a well known
 characterization of integrable mappings.
 \begin{proposition} \label{lem:jac} Consider a multi-leader multi-follower
 game in which the objective functions $\varphi_i, i \in \Nscr$ of the
 leaders are continuously differentiable. Then the game is a potential
 game if and only if for all $(x,y) \in X \times Y$, the Jacobian $\nabla
 F(x,y)$ is {a} symmetric matrix.
 \end{proposition}

\section{Existence statements for the shared-constraint formulation}\label{sec:shared}
In this section, 	we present a shared constraint modification of
the conventional formulation and existence results for it. We begin with
an illustrative example in Section~\ref{sec:anotherpf} and provide a general
formulation in Section~\ref{sec:32}. Sufficiency conditions for the
existence of global and Nash-stationary equilibria are derived in
Sections~\ref{sec:33} and ~\ref{sec:34}, respectively. The section
concludes with Section~\ref{sec:35} which provides an analysis of existence of global equilibria via
fixed-point theory. 

\subsection{Motivation: The Pang and Fukushima example~\cite{pang05quasi}} \label{sec:anotherpf}
To motivate the modified model we recall the example Pang and Fukushima~\cite{pang05quasi} presented to make the point that even simple \mlmfgs may not admit equilibria in pure strategies. We then analyze a modified version of this example that captures the spirit of the modified formulation we present. 
\begin{examplec}[A modified version of the Pang and
Fukushima example~\cite{pang05quasi}:]\label{ex:pfshared}
Pang and Fukushima consider a multi-leader multi-follower game comprising of two leaders
 and one follower~\cite{pang05quasi}.  The follower is solves the optimization problem
$$ \min_{y \geq 0} \left\{y(-1+x_1+x_2) + \half y^2\right\} = \max
\left\{ 0, 1-x_1-x_2\right\}$$ 
Leaders solve the following optimization problems.
$$	\problemsmall{L$_1(x_2)$}
	{x_1,y_1}
	{\varphi_1(x_1,y_1) = \half x_1 +  y_1 }
				 {\begin{array}{r@{\ }c@{\ }l}
	x_1 &\in& [0,1] \\
y_1 &=& \max \{ 0, 1-x_1-x_2\}
\end{array}}\ \problemsmall{L$_2(x_1)$}
	{x_2,y_2}
	{\varphi_2(x_2,y_2) = -\half x_2 -y_2 }
				 {\begin{array}{r@{\ }c@{\ }l}
x_2 & \in & [0,1] \\
y_2 &=& \max \{0, 1-x_1-x_2\}	
\end{array}} $$
where $X_1 = X_2 = [0,1]$ and $Y = \Real$. 
By substituting for $y_1$ (respectively, $ y_2 $), we find that L$_1(x_2)$ is a convex problem for any $x_2$ but  L$_2(x_1)$ is not a convex problem in the space of $ x_2 $. Specifically, L$_2(x_1) $ can be
rewritten as 
$$\problemsmall{L$_2(x_1)$}
	{x_2}
	{\min \left(-\half x_2, -1+x_1 -\half x_2\right)}
				 {\begin{array}{r@{\ }c@{\ }l}
x_2 & \in & [0,1]. 	
\end{array}} $$
The reaction maps $ \Rscr_1:X_2 \rightarrow X_1,\Rscr_2: X_1 \rightarrow X_2 $ that capture the best response for players $ 1,2 $ in the $(x_1,x_2)$ space are given by the
following: 
 \begin{align*}
 \Rscr_1(x_2) &= \{1-x_2\}\ \ \forall x_2 \in [0,1] \qquad \mbox{ and } \qquad \Rscr_2(x_1) = \begin{cases}
                                              \{0\} & x_1 \in [0,\half)\\
\{0,1\} & x_1 = \half \\
\{1\} & x_2 \in (\half,0].   \end{cases}
\end{align*}
It is easy to see that $\Rscr \triangleq \Rscr_1\times \Rscr_2$ has no fixed point whereby this game has no equilibrium. 
Finally, note that this is game is a potential game and it admits a potential function in the $(x,y)$ space given by: 
\[\pi(x,y) = \varphi_1(x_1,y_1) + \varphi_2(x_2,y_2) = \half x_1 +  y_1 -\half x_2 -y_2.\]


Consider the following modification of this example. Leader $1$ has an additional constraint, `$y_2 = \max \{0, 1-x_1-x_2\}$', which in
the original problem appeared in leader $2$'s optimization problem. Likewise, leader
$2$ now has an additional constraint `$y_1 = \max \{0, 1-x_1-x_2\}$', which
in the original problem,  was in leader $1$'s optimization problem. More
specifically,
\begin{itemize}
	\item \textit{both leaders} are constrained by \textit{both
		equilibrium constraints};
	\item Leader $i$'s problem is parametrized  the decisions of {\em rival
	leaders} (denoted by $x\mi$) and {\em the other leader's
conjectures about the follower equilibrium} (denoted
			by $y\mi$).
\end{itemize}
$$	\problemsmall{L$_1(x_2,y_2)$}
	{x_1,y_1}
	{\varphi_1(x_1,y_1) = \half x_1 +  y_1 }
				 {\begin{array}{r@{\ }c@{\ }l}
	x_1 &\in& [0,1] \\
y_1 &=& \max \{ 0, 1-x_1-x_2\}\\
y_2 &=& \max \{0, 1-x_1-x_2\}	
\end{array}}\ \problemsmall{L$_2(x_1,y_1)$}
	{x_2,y_2}
	{\varphi_2(x_2,y_2)= -\half x_2 -y_2 }
				 {\begin{array}{r@{\ }c@{\ }l}
x_2 & \in & [0,1] \\
y_1&=&\max\{0, 1-x_1-x_2\}\\
y_2 &=& \max \{0, 1-x_1-x_2\}	
\end{array}} $$
We claim that $((x_1,x_2),(y_1,y_2))= ((0,1),(0,0))$ is an equilibrium of this modified game.
To see why this is true,  observe that  Leader $1$ gets $\varphi_1(0,0) =
0$ whereas leader $2$ gets $\varphi_2(1,0)=-\half$.  Leader $1$'s
global minimum is $0$ and he thus has no incentive to deviate from this strategy. Leader $2$'s strategy set at equilibrium reduces to a singleton
containing only his equilibrium strategy. This is induced by the
presence of leader $1$'s equilibrium constraint in his optimization
problem (the constraint $y_1 = \max \{0, 1-x_1,x_2\}$ is, at
		equilibrium, equivalent to $ 0 = \max \{0, 1-x_2\}$; together
		with the constraint $x_2 \in [0,1]$ this implies $x_2 =1$ and
		$y_2=0$.)  
\end{examplec}

In the following section we generalize the approach adopted in this example. The modified game has shared constraints even while the original does not. We show (Theorem \ref{thm:pot}) that a {potential game} with
shared constraints admits an equilibrium under mild conditions. Indeed global
minimizers of the potential function over the shared constraint are
equilibria of this game. In this game, the shared constraint is given by
the set $$\Fscr^{\ae} =
\left\{(x_1,x_2,y_1,y_2) \left| (x_1,x_2) \in [0,1]^2, (y_1,y_2) \geq 0, \   \ \begin{aligned}
		y_1 =\max \{ 0,
	1-x_1-x_2\} \\
	y_2 =\max \{ 0,
	1-x_1-x_2\} 
		\end{aligned} \right. \right\}.$$ 
	We will explain the notation $\Fscr^{\ae}$ in the following sections. 
The global minimizer of $\pi$ over $\Fscr^{\ae}$ is 
\begin{align*}\arg\min_{(x,y)\in \Fscr^{\ae}} \half x_1 +  y_1 -\half x_2 -y_2&= ((0,1),(0,0)),
\end{align*}
which is indeed the equilibrium.

\subsection{Modification: Leaders sharing all equilibrium
	constraints}\label{sec:32}
Consider the formulation in which the $i^{\rm th}$ leader solves the following optimization problem. 
$$
	\problem{L$^{\ae}_i(x^{-i},y^{-i})$}
	{x_i,y_i}
	{\varphi_i(x_i,y_i;x^{-i})}
				 {\begin{array}{r@{\ }c@{\ }l}
		x_i &\in& X_i, \\
		y_i &\in & Y_i, \\
		y_j &\in& \Sscr(x), \qquad j =1,\hdots,N.
 	\end{array}}
	$$
We  denote this game by $\Epec^{\ae}$ and note that  
the difference between $\Epec^{\ae}$ and $\Epec$ is that {\em all} constraints 
$y_j \in \Sscr(x),$ $j=1,\hdots,N$ are now a part of {\em each} leader's 
optimization problem. In effect, each leader takes into account the
conjectures regarding the follower equilibrium made by all other
leaders.  The result is that for any $i$, $y_i$ satisfies the same
constraints in problems L$_i$ and L$_i^{\ae}$, but $x_i$ is constrained
by additional constraints in L$_i^{\ae}$.  


\gap

For $y_j \in Y_j, x_j \in X_j$ for $j\neq i$, let
$\Omega_i^{\ae}(x^{-i},y^{-i})$ be the feasible region of
L$_i^{\ae}(x^{-i},y^{-i})$ and let $\Omega^{\ae}, \Fscr^{\ae}, \Sscr^N$
and $\Gscr$ be defined as 
\begin{align}
\Omega^{\ae}(x,y) &\triangleq \prod_{i=1}^N \Omega_i^{\ae}(x^{-i},y^{-i}), \quad &\Fscr^{\ae} &\triangleq \{ (x,y) \ | \ (x,y) \in \Omega^{\ae}(x,y) \}, \\ 
\Sscr^N(x)&\triangleq \prod_{i=1}^N \Sscr(x), \quad &\Gscr &\triangleq \{(x,y) \ | y \in  \Sscr^N(x)\}, \label{eq:sscrn}
\end{align}
where $\Gscr$ is the graph of $\Sscr^N$ and $\Fscr^{\ae}$ is the set of fixed points of $\Omega^{\ae}$. An equilibrium of $\Epec^{\ae}$ is a point 
$$ (x,y) \in \Fscr^{\ae},\ \mbox{such that} \quad \varphi_i(x_i,y_i;x^{-i}) \leq \varphi_i(\bar{x}_i,\bar{y}_i;x^{-i}) \qquad \forall \ (\bar{x}_i,\bar{y}_i) \in \Omega^{\ae}_i(x^{-i},y^{-i}), \forall i.$$
\begin{proposition}\label{eq-orig-ae}
Consider the multi-leader multi-follower game defined by $\Epec^{\ae}.$
Then the following hold:
\begin{enumerate}
	\item[(i)] The mapping $\Omega^{\ae}(x,y)$ is a shared constraint mapping
		satisfying \eqref{eq:shared};
	\item[(ii)] A point $(x,y)$ is a fixed point of $\Omega^{\ae}$ if and only
		if it is a fixed point of $\Omega$. \ie, $ \Fscr= \Fscr^{\ae} $;
	\item [(iii)] Every equilibrium of $\Epec$ is an equilibrium of $\Epec^{\ae}$. \label{lem:ae}
	\item[(iv)] $ \FscrS $ is a closed set.
\end{enumerate}
\end{proposition}
\begin{proof}
	\begin{enumerate}
		\item[(i)] It {can be seen that for any $i$ and any
		$x\mi,y\mi$, where  $x_j \in X_j,y_j \in Y_j$ for all $j \neq
			i$, we have that}
\begin{align*}
	\Omega_i^{\ae}(x^{-i},y^{-i}) & = {\{x_i,y_i \ | \ x_i \in X_i, y_i
	\in Y_i, y_j \in \Sscr(x) \mbox{ for }  j= 1, \hdots, N\}} \\
			& = {\{x_i,y_i \ | \ x_i \in X_i, y_i
	\in Y_i, y \in \Sscr^N(x)\}.} 
\end{align*}
But $y_j \in Y_j, x_j \in X_j$ for $j \neq i$, implying that 
\begin{align*}
\{x_i,y_i \ | \ x_i \in X_i, y_i
\in Y_i, y \in \Sscr^N(x) \}
& =\{x_i,y_i \ | \ x_i \in X_i, {y_j
	\in Y_j \mbox{ for } j = 1, \hdots, N}, y \in \Sscr^N(x)\} \\
			& = \{x_i,y_i \ | \ x \in X, y \in Y, (x,y) \in \Gscr\}, 
\end{align*}
where $\Gscr$ is defined in \eqref{eq:sscrn}.
Thus $\Omega^{\ae}$ is a shared constraint of the form dictated by \eqref{eq:shared}. 
		\item[(ii)] {It suffices to show that $\Fscr = \Fscr^{\ae}$.
			But, from (i) it follows that $\Fscr^{\ae} =   (X\times
					Y)\cap \Gscr$. It is easy to see from the definition of $\Fscr$ that 
				$ \Fscr=(X\times Y)\cap \Gscr$. The result follows.} 
		\item[(iii)] An equilibrium $(x,y)$ of $\Epec$ lies in $\Fscr$ and thereby in $\Fscr^{\ae}$. Since $\Omega^{\ae}_i(x^{-i},y^{-i}) \subseteq \Omega_i(x^{-i},y^{-i})$, 
the result follows. 
\item[(iv)] The relation $ y \in \Sscr^N(x) $ is equivalent to the set of equations
\[ \fnat(y_i;x) =0, \quad \forall i \in \Nscr\]
where $ \fnat(\cdot;x) $ is the natural map~\cite{facchinei03finiteI} of $ \VI(G(x,\cdot),K(x))$. Since $ K,G $ have been assumed continuous (cf. immediately following \eqref{eq:sscr}), it follows that the zeros of $ \fnat $ form a closed set.
	\end{enumerate}
\end{proof}

A special case of the conventional game $ \Epec $ which is already a shared-constraint game is the case (denoted by $\Epec^{\rm bl}$) where leaders have disjoint set of followers. Effectively, each leader solves bilevel optimization problems as follows.
$$
	\problem{L$^{\bl}_i(x^{-i})$}
	{x_i,y_i}
	{\varphi_i(x_i,y_i;x^{-i})}
				 {\begin{array}{r@{\ }c@{\ }l}
		x_i &\in& X_i, \\
		y_i &\in& \widehat{\Sscr}_i(x_i),\\
		y_i &\in& Y_i. 
	\end{array}}
	$$
{Since $y_i \in \widehat{\Sscr}_i(x_i)$, there is no coupling of
	leader decisions in the constraints of leader problems}. This is a
	special case of $\Epec$ with $\Sscr(x) \equiv \prod_{i\in \Nscr}
	\widehat{\Sscr}_i(x_i)$, where $\widehat{\Sscr}_i(x_i)$ 	is the
	solution of a variational inequality for each $i$
	and where the objective of leader $ i $ depends only on the equilibrium of $ \widehat{\Sscr}_i(x_i)$ and not on 
	$ \widehat{\Sscr}_j(x_j) $ for $ j \neq i. $ With a slight abuse of our notation 	so far, we let $ y_i $ denote an element of the set $ \widehat{\Sscr}_i(x_i) $ and $ Y_i $ be the space of such $ y_i. $ 	Let
	$\Omega^{\bl}_i$ be the feasible region of L$^{\bl}_i(x\mi,y\mi)$
	and let $\Fscr^{\bl}$ be the set of fixed points of $\Omega^{\bl}
	\triangleq \prod_{i=1}^N \Omega_i^{\bl}$. Since there is no
	coupling, it is easily seen that 
$$\Fscr^{\bl} = \Omega^{\bl} = \{ (x,y) \ | \ x \in X, y \in \tilde Y, (x,y) \in \widehat{\Gscr}\}, $$
where $\widehat{\Gscr} = \prod_{i=1}^N \widehat{\Gscr}_i$, $\tilde Y
\triangleq \prod_{i=1}^N Y_i$, and $\widehat{\Gscr}_i$ is the graph of $\widehat{\Sscr}_i$. 
If $(x_j,y_j) \in \Omega^{\bl}_j$, for $j \neq i$, 
$$\Omega_i^{\bl} = \{ (x_i,y_i) \ | \ {x_i \in X_i, y_i \in Y_i ,y_i
	\in \widehat{\Sscr}_i(x_i)}\} = \{ (x_i,y_i)\ | \ (x,y) \in \Fscr^{\bl}\}. $$
It follows that this game is a shared constraint game.

%

%

\subsection{Existence of global equilibria}\label{sec:33}
We now present existence results for the games $ \Epec^{\ae} $, but since the only property we use is the 
shared constraint structure, our results apply also to $ \Epec^{\bl}. $
We emphasize that  in our modified formulation $ \Epec^{\ae} $, the
optimization   problem of each leader is indeed constrained by an  
equilibrium constraint -- and it is thus a hard nonconvex problem in its own right.

Our main result relates the global minimizers of the following optimization problem to the equilibria of $ \Epec^{\ae}.$
 $$\problem{P$^{\rm ae}$}
	{x,y}
	{\pi(x,y)}
				 {\begin{array}{r@{\ }c@{\ }l}
(x,y) \in \FscrS.
	\end{array}}
	$$
\begin{theorem}[Minimizers of $\PS$ and Equilibria of $\EpecS$] \label{thm:pot}
Let $\EpecS$ be a  potential multi-leader multi-follower game with a
potential function $\pi$.  Then any global minimizer of $\pi$ over $\FscrS$ is an equilibrium of $\EpecS$. 
\end{theorem}
\begin{proof}
Let $(x,y) \in \FscrS$ be a global minimum of $\pi$ over $\FscrS$. Then,  for each $i \in \Nscr$
\begin{align*}
\pi(x_i,y_i,x\mi,y\mi) - \pi(u_i,v_i,x\mi,y\mi) &\leq 0 \quad \forall \ (u_i,v_i) : (u_i,v_i,x\mi,y\mi) \in \FscrS.  
\end{align*}
But, $(u_i,v_i,x\mi,y\mi) \in \FscrS$ if and only if $(u_i,v_i) \in \OmegaS_i(x\mi,y\mi),$ since $\OmegaS$ is a shared constraint. Using this, together 
with the fact that $\pi$ is a potential function, we obtain that for
	each $i$
\[\varphi_i(x_i,y_i;x\mi,y\mi) - \varphi_i(u_i,v_i;x\mi,y\mi) \leq 0 \quad\forall \ (u_i,v_i) \in \OmegaS_i(x\mi,y\mi).\]
This implies that for $i = 1, \hdots, N$, given $(x^{-i},y^{-i})$,
	the vector  $(x_i,y_i)$ lies in the set of best responses for leader
		$i$. In other words, $(x,y)$ is an
		equilibrium of $\EpecS$. 
\end{proof}

It now follows that if the minimizer of $\PS$ exists, the game $\EpecS$ admits an equilibrium. 
\begin{theorem}[Existence of equilibria of $\EpecS$] \label{thm:pot2}
Let $\EpecS$ be a  potential multi-leader multi-follower game with a
potential function $\pi$. Suppose $\FscrS$ is a nonempty set and $\varphi_i(x)$ is a continuous function for $i = 1, \hdots, N$. If the minimizer of $\PS$ exists
(for example, if either $\pi$ is a coercive function on $\FscrS$ or if
 $\FscrS$ is compact), then $\EpecS$ admits  an equilibrium. 
\end{theorem}
\begin{proof}
 It is easy to see from \eqref{eq:pot} that $\pi$ is continuous. By
	 the hypothesis of the theorem, $\pi$ achieves its global minimum on
		 $\FscrS$. This could, for instance, be deduced from the
		 coercivity of $\pi$ over a nonempty set $\FscrS$ or {by
			 the} compactness of $\FscrS$ {Based on
				 Theorem~\ref{thm:pot}},  a global minimizer of $\pi$ is
				 an equilibrium of $\EpecS$ and the result follows. 
 \end{proof}

\noindent \textbf{Remark:\ }Recall that $ \Fscr^{\ae} =\Fscr.$ Therefore $ \PS $ is essentially a minimization of $ \pi $ over $ \Fscr. $ 
Furthermore, since the closedness of $ \Fscr $ is already established, compactness follows from the boundedness of $ X \times Y. $ \hfill $\square$

If the objectives of the leaders are independent of the strategies of
	other leaders, the sum of the objectives is a potential function, whereby any such game is a potential game. We thus have the following corollary.
\begin{corollary} \label{thm:new}
Consider a multi-leader multi-follower game $\EpecS$ for which 
$\FscrS$ is nonempty and $\varphi_i, i\in \Nscr$ are continuous. 
 Assume further that for each $i \in \Nscr$, $\varphi_i(x_i,y_i;x^{-i}) \equiv
\varphi_i(x_i,y_i)$, \ie, assume that $\varphi_i$ is independent of
$x^{-i}$. If, either the functions $\varphi_i, i \in \Nscr$ are coercive or if $\FscrS$ is compact, 
the game $\EpecS$ has an equilibrium.
\end{corollary}
\begin{proof}
If $\varphi_i(x_i,y_i;x^{-i}) \equiv \varphi_i(x_i,y_i)$ for each $i$,
   $\pi = \sum_{i\in \Nscr}\varphi_i$ is a potential function. Then by
   Theorem \ref{thm:pot}, the game has an equilibrium. 
   \end{proof}


	\subsection{Existence of Nash stationary equilibria}\label{sec:34}
Since $ \Fscr^{\ae} = \Fscr $ is characterized by equilibrium constraints, $ \PS $ is an MPEC.
	In this section,  we relate stationary
		points and local minimizers $ \PS $ to {their equilibrium counterparts in the context of}
		$\EpecS$. 
		These relations assume relevance because, being an MPEC, the global minimization of $\pi$ over $\FscrS$ is
	hindered by the nonconvexity of $\FscrS$ as well as the possible
	nonconvexity of $\pi$. When solved computationally, standard nonlinear
	programming solvers may only produce a suitably defined stationary
	point of ${\rm P}^{\ae}$. Traditionally, while a range of
	stationarity points are considered in the context of mathematical programs with equilibrium constraints~\cite{scheel00mathematical}, we
	focus on the notions Bouligand stationarity, local
		minima, strong stationarity and second-order strong stationarity. The proofs of these results are quite similar to those from our recent submission~\cite{kulkarni13existence}; we therefore provide only a sketch of each proof. 	
\subsubsection{{B-stationary} equilibria} \label{sec:local}

	\begin{definition}[Nash B-stationary point]
	A point $(x,y) \in \FscrS$ is a {\em Nash B-stationary point} of $\EpecS$ if for all $i \in \Nscr$,
	\[\nabla_i\varphi_i(x,y)\t d \geq 0 \quad \forall d \in \Tscr((x_i,y_i);\OmegaS_i(x\mi)), \]
	where $\Tscr(z;K)$, the tangent cone at $z \in K \subseteq
			\mathbb{R}^n$, is defined as follows:
			$$ \Tscr(z;K) \triangleq \left\{ dz \in \Real^n: \exists
			\{\tau_k\}, \{z_k\} \mbox{ such that }  dz = \lim_{k \to \infty}
			\left(\frac{z_k - z}{\tau_k}\right),  K \ni 
			z_k \to z, 0 < \tau_k \to 0  \right\}. $$
	\end{definition}	
	
	\begin{proposition} [B-Stationary points of P$^{\ae}$ and
	Nash B-stationary points of $\EpecS$]
Let $\EpecS$ be a potential \mlmfg with potential function $ \pi $ and suppose $\{\varphi_i\}_{i \in
		\Nscr}$ are
	continuously differentiable functions over $X \times Y$. If $(x,y)$ is a B-stationary point of
	$\PS$, then $(x,y)$ is a Nash  B-stationary point of $\EpecS$.
	\label{prop:statae}
	\end{proposition}
	\begin{proof}
A stationary point $(x,y)$ of $\pi$ over $\FscrS$ satisfies 
	\begin{align} \label{stat-pi}
		\nabla_{x}  \pi(x,y)\t d{x}  + \nabla_{y}
			\pi(x,y)\t  d{y} \geq 0, \quad \forall (d{x},d{y}) \in
			\Tscr ((x,y);\FscrS).
		\end{align}
		Fix some $i\in \Nscr$ and consider an arbitrary $(dx_i',dy_i') \in
	\Tscr(x_i,y_i;\OmegaS_i(x^{-i},y^{-i})).$ Let $( u_{i,k},v_{i,k}) \in \OmegaS(x^{-i}, y^{-i})$,  $ ( u_{i,k},v_{i,k}) \buildrel{k}\over\rightarrow (x_i,y_i)$ and $0<\tau_k \buildrel{k}\over\rightarrow0$ such that $\frac{u_{i,k} -x_i}{\tau_k} \buildrel{k}\over\rightarrow dx_i'$ and $ \frac{v_{i,k}-y_i}{\tau_k}\buildrel{k}\over\rightarrow dy_i'$. It follows that the sequence 
$(\textbf{x}_{i,k},\textbf{y}_{i,k}) \in \FscrS$,
		where 
		\begin{align}\label{def-zi2}
			\textbf{x}_{i,k} = (x_1, \hdots, u_{i,k}, \hdots, x_N), \aur  
			\textbf{y}_{i,k} = (y_1, \hdots, v_{i,k}, \hdots, y_N). 
		\end{align}
 Therefore, the direction $(\textbf{dx}_i,\textbf{dy}_i) \in \Tscr(z;\FscrS)$ where 
		$$ \textbf{dx}_i = (0, \hdots, dx_i', \hdots, 0) \mbox{ and } \textbf{dy}_i =
		(0, \hdots, dy_i', \hdots, 0). $$
Taking $(dx,dy) = (\textbf{dx}_i,\textbf{dy}_i)$ in \eqref{stat-pi} and using \eqref{eq:potdiff} we get 
the required result.
	\end{proof}

\begin{proposition}\textbf{\em (Local minimum of $\PS$ and local Nash equilibrium)}
Consider the \mlmfg $\EpecS$ with potential function $\pi$. If $(x,y)$ is a local minimum of $\PS$, 
then $(x,y)$ is a local Nash equilibrium of $\EpecS$.
\end{proposition}
\begin{proof}
The proof is analogous to that of Theorem \ref{thm:pot}.  If $(x,y)$ is
a local minimum of $\PS$, there exists a neighborhood of $(x,y)$, denoted
by $\Bscr(x,y)$, such that 
\begin{equation}
\pi(x,y) \leq \pi(x',y'), \qquad \forall\ (x',y') \in \Bscr(x,y) \cap \FscrS. \label{eq:loc}
\end{equation}
Consider an arbitrary $i \in \Nscr$ and let $\Bscr_i(x_i,y_i;x\mi,y\mi)
:= \left\{(u_i,v_i)\ |  (u_i,v_i,x\mi,y\mi) \in \Bscr(x,y)\right\}$.  Then it follows that
$$(u_i,v_i) \in \left(\OmegaS_i(x\mi,y\mi) \cap
		\Bscr_i(x_i,y_i;x\mi,y\mi)\right)
 \iff (u_i,v_i,x\mi,y\mi) \in \left(\FscrS \cap \Bscr(x,y)\right).$$ Thus, using this relation in \eqref{eq:loc} and employing 
\eqref{eq:pot}, we get 
\[\varphi_i(x,y) \leq \varphi_i(u_i,v_i,x\mi,y\mi), \qquad \forall\
		(u_i,v_i) \in \left(\OmegaS_i(x\mi,y\mi) \cap
				\Bscr_i(x_i,y_i;x\mi,y\mi)\right).\]
In other words, $(x_i,y_i)$ is a local minimizer of $\LiS(x\mi,y\mi)$. This holds for each $i \in \Nscr,$ whereby $(x,y)$ is a local Nash equilibrium. 
\end{proof}
\subsubsection{Strong stationarity}
Here we relate other notions of stationarity for P$^{\ae}$ with weaker equilibrium notions of $ \Epec^{\ae}. $ When the algebraic form of the constraints are available, a \textit{strong-stationary point} can be
	defined. Let $X_i = \{x_i | c_i(x_i) \geq 0\},Y_i = \{y_i | d_i(y_i)
	\geq 0\}$, where $c_i,d_i$ are continuously differentiable. Let
	$\Sscr(x)$ be the solution of a complementarity problem: $y_i
	\in \Sscr(x) \iff 0 \leq y_i \perp G(y_i;x) \geq 0$, and $G$ is $\Real^p$-valued and continuously differentiable. Thus $\PS$ can
	be written as
	$$\problem{$\PS$}
	{x,y}
	{\pi(x,y)}
	{\begin{array}{r@{\ }c@{\ }l}
		\left\{ \begin{aligned}
			c_i(x_i) & \geq  0 \\
			d_i(y_i) & \geq  0 \\
		0 \leq y_i & \perp  G(y_i,x) \geq 0 
				\end{aligned}\right\}\quad i =1,\hdots, N.
	\end{array}}	
$$
To define the stationarity  conditions, we  define the relaxed nonlinear
	program below which requires specifying the index sets 	{$\tilde
	\Iscr_{1i}$ and $\tilde \Iscr_{2i}$ for $i=1,
	\hdots, N$ where $\tilde \Iscr_{1i}, \tilde \Iscr_{2i} \subseteq \{1, \hdots,
	p\}$ and $ \tilde\Iscr_{1i} \cup  \tilde\Iscr_{2i} = \{1, \hdots, p\}$.}
	$$\problem{$\PS_{rnlp}$}
	{x,y}
	{\pi(x,y)}
	{\begin{array}{r@{\ }c@{\ }l}
\left\{ \begin{aligned}
			c_i(x_i) & \geq  0 \\
			d_i(y_i) & \geq  0 \\
			{[y_i]}_j & =  0, \quad  \forall j  \in \tilde \Iscr_{2i}^{\perp} \\
			{[G(y_i,x)]}_j & =  0, \quad   \forall j \in  \tilde\Iscr_{1i}^{\perp}  \\
			{[y_i]}_j & \geq  0, \quad   \forall j \in  \tilde\Iscr_{1i}  \\
			{[G(y_i,x)]}_j & \geq  0, \quad \forall j \in  \tilde\Iscr_{2i}  
		\end{aligned}\right\}\quad i =1,\hdots, N,
	\end{array}}
	$$
where $[\cdot]_j$ denotes the $j\th$ component of `$\cdot$' and
{$ \tilde\Iscr_{1i}^\perp, \tilde\Iscr_{2i}^\perp$ denote the complements of
	$ \tilde\Iscr_{1i}, \tilde\Iscr_{2i}$ respectively.} {Further, we refer to the both
		index sets collectively as $\tilde \Iscr_i$ and  the collection
		of index sets $\{\tilde \Iscr_1, \hdots, \tilde \Iscr_N\}$ by
		$\tilde \Iscr$.} Note that
	in accordance with~\cite{fletcher02local}, we define the index sets
	independent of the point $(x,y)$. 
	 We may now state the strong stationarity conditions at a particular point $(x,y)$. 

	\begin{definition}[Strong-stationarity point of $\PS$]
A point $(x,y) \in \FscrS$
		is a  {\em strong stationarity point} of $\PS$ if there exist Lagrange
		multipliers $\eta_i, \mu_i, \lambda_i$ and $\beta_i, i \in \Nscr$ such that the following
		conditions hold:
		\begin{align}\label{ss-pae}
		\left\{
		\begin{aligned}
\nabla_{x_i} \pi(x,y) - \nabla_{x_i} c_i(x_i)\t \eta_i -
\sum_{k=1}^N\nabla_{x_i} G(y_k,x)\t\beta_k   &= 0 \\
\nabla_{y_i} \pi(x,y)  - \nabla_{y_i} d_i(y_i)\t \mu_i - \lambda_i -
\nabla_{y_i} G(y_i,x)\t\beta_i   & = 0 \\
0 \leq \eta_i & \perp  c_i(x_i)  \geq  0 \\
0 \leq \mu_i &\perp	d(y_i)  \geq  0 \\
			{y_i} & \geq  0, \\
{[\lambda_i]}_j {[y_i]}_j & = 0, \quad \forall j \\
			{G(y_i,x)} &  \geq  0, \\ 
{[\beta_i]}_j {[G(y_i,x)]}_j & = 0, \quad \forall j \\
{[y_i]}_j  = 0 \mbox{ or }   {[G(y_i,x)]}_j & {=0},
			\quad  \, \forall j  \\
		\mbox{ if } {[G(y_i,x)]}_j  = 0 \mbox{ and } {[y_i]}_j  = 0,
			\mbox{ then } [\lambda_i]_j, [\beta_i]_j & \geq 0, \quad
			\forall j
		\end{aligned}\right\}, \quad &\forall i \in \Nscr.
	\end{align}
	\end{definition}
	Having defined the strong stationarity conditions, we are now in a
	position to define the second-order sufficiency conditions. These
	assume relevance in defining a local Nash equilibrium; loosely
	speaking, at a local Nash equilibrium,  every agent's decision satisfies the \textbf{m}athematical \textbf{p}rograms with \textbf{e}quilibrium \textbf{c}onstraints-second-order
	sufficiency or the MPEC-SOSC conditions, given the decisions of its competitors. Furthermore, corresponding to a
	stationary point of $\PS_{rnlp}$, we may prescribe an active set
	$\tilde \Ascr(x,y)$ such that $\tilde \Ascr(x,y) \triangleq \{\tilde \Ascr_1(x,y), \hdots,
	\tilde \Ascr_N(x,y)\}$, where
	$\tilde \Ascr_i(x,y)$ denotes the set of active
	constraints corresponding to the set of constraints
	$$ \left\{ \begin{aligned}
			c_i(x_i) & \geq  0 \\
			d_i(y_i) & \geq  0 \\
			{[y_i]}_j & =  0, \quad  \forall j  \in \tilde \Iscr_{2i}^{\perp} \\
			{[G(y_i,x)]}_j & =  0, \quad   \forall j \in \tilde \Iscr_{1i}^{\perp}  \\
			{[y_i]}_j & \geq  0, \quad   \forall j \in \tilde \Iscr_{1i}  \\
			{[G(y_i,x)]}_j & \geq  0, \quad \forall j \in \tilde \Iscr_{2i}  
		\end{aligned}\right\}.$$
	Suppose $\tilde \Ascr_i(x,y) = \{\tilde
	\Ascr_i^c(x,y),\tilde 
	\Ascr_i^d(x,y), \tilde \Ascr_i^e(x,y)\}$, where $\tilde \Ascr_i^c,
	\tilde \Ascr_i^d$ and $\tilde \Ascr_i^e$ denote the active sets
	associated with $c_i(x_i) \geq 0$, $d_i(y_i) \geq 0$, and the
	remaining constraints, respectively. The specification of the active
	set allows us to define the 
	critical cone $S^*(x,y)$ as 
	\begin{align}\label{feas-dir} S^*(x,y) \triangleq \left\{s: s \neq 0, \nabla \pi(x,y)\t s = 0, a_j\t s =
	0, j \in \tilde \Ascr(x,y), a_j\t s \geq 0, j \not \in \tilde \Ascr(x,y)\right\},
	\end{align}
	where  $a_j$ denotes the constraint gradients of the $j\th$
		constraint.
	\begin{definition} [Second-order Strong-stationarity point of $\PS$]
		A point $(x,y)$ of the optimization problem $\PS$ is a {\em second-order strong stationary point} of $\PS$ if it is a strong
		stationary point with Lagrange 	multipliers $(\eta,\mu,\lambda,\beta)$ 
				and $s^T\nabla^2_{x,y} \Lscr\ s > 0$ for
		$s \in S^*(x,y)$, where $S^*(x,y)$ is given by \eqref{feas-dir} and
		$\nabla^2_{x,y}\Lscr$ denotes the Hessian of the Lagrangian
		of $\PS_{rnlp}$ with respect $(x,y)$ evaluated at 
				$(x,y,\eta,\mu,\lambda,\beta).$ 
	\end{definition}

 	Next, we provide a formal definition of
	\textit{Nash strong-stationary} and \textit{Nash second-order strong-stationary
	points} of $\EpecS$, which requires defining the
critical cone $S^*_i(x,y)$ for each leader $i = 1,\hdots, N$: 
\begin{align}\label{feas-dir-Li} S_i^*(x,y) \triangleq \left\{s_i: s_i
\neq 0, \nabla_{i} \varphi_i(x,y)\t s_i = 0, a_j\t s_i =
	0, j \in \Ascr_i(x,y), a_j\t s_i \geq 0, j \not \in \Ascr_i(x,y)\right\},
	\end{align}
where  {$\Ascr_i(x,y)$\footnote{{The active set associated with
	$\PS_{rnlp}$ is denoted by $\tilde \Ascr$ while the active set
		associated with leader $i$'s problems is denoted by $\tilde
		\Ascr_i$.}}  denotes the active set utilized in defining the
	relaxed nonlinear program associated with $\LiS(x\mi,y\mi)$}	and
	$a_j$ denotes the constraint gradient associated with $j\th$
	constraint.

	\begin{definition} [Nash strong-stationary and Nash second-order strong-stationary
	points] A point $(x,y) \in \FscrS$ is a Nash strong-stationary point of $\EpecS$ if for $i=1,
	\hdots, N$, 
		there exist Lagrange multipliers  $\bar{\eta}_i, \bar{\mu}_i,
			\bar{\lambda}_i$ and $\bar{\beta}_i^k, k =1,\hdots,N$, 
such that the following conditions hold:
		\begin{align}\label{ss-Li}
		\left\{
		\begin{aligned}
\nabla_{x_i} \varphi_i(x,y) - \nabla_{x_i} c_i(x_i)\t \bar{\eta}_i -
\sum_{k=1}^N\nabla_{x_i} G(y_k,x)\t\bar{\beta}_i^k   &= 0  \\
\nabla_{y_i} \varphi_i(x,y)  - \nabla_{y_i} d_i(y_i)\t \bar{\mu}_i - \bar{\lambda}_i - \nabla_{y_i} G(y_i,x)\t\bar{\beta}^i_i   &= 0 \\
0 \leq \bar{\eta}_i & \perp  c_i(x_i)  \geq  0 \\
0 \leq \bar{\mu}_i &\perp	d(y_i)  \geq  0 \\
			{y_i} & \geq  0, \\
{[\bar{\lambda}_i]}_j {[y_i]}_j & = 0, \quad \forall j \\
			{G(y_i,x)} &  \geq  0, \\ 
{[\bar{\beta}^k_i]}_j {[G(y_i,x)]}_j &=0, \quad \forall k \in \Nscr, \forall j  \\
			{[y_i]}_j  = 0 \mbox{ or }   {[G(y_i,x)]}_j & {=0},
			\quad  \forall j  \\
					\mbox{ if } {[G(y_i,x)]}_j  = 0 \mbox{ and } {[y_i]}_j  = 0,
			\mbox{ then } [\bar{\lambda}_i]_j, [\bar{\beta}^k_i]_j &\geq 0, \quad
\forall k \in \Nscr,			\forall j \\
		\end{aligned}\right\}, &\qquad \forall i \in \Nscr.
	\end{align}
	Furthermore, $(x,y)$ is a Nash second-order strong stationary point of
	$\EpecS$ if $(x,y)$ is a Nash strong stationary point of
	$\EpecS$  and if for $i = 1, \hdots, N$, $s_i^T \nabla^2_{x_i,y_i}
{\Lscr}_i(x,y) s_i > 0$ for $s_i \in S_i^*(x,y)$ where $S_i^*(x,y)$ is given by
\eqref{feas-dir-Li}, where $\nabla^2_{x_i,y_i} \Lscr_i$ denotes the
Hessian of the Lagrangian function of
$\LiS(x^{-i},y^{-i})$ with respect to $(x_i,y_i)$ evaluated at
$(x_i,y_i,\bar{\eta}_i,\bar{\mu}_i,\bar{\lambda}_i,\bar{\beta}^1_i,\hdots,
		\bar{\beta}^N_i)$ if $\EpecS=\Epec^{\ae}$ or at
$(x_i,y_i,\bar{\eta}_i,\bar{\mu}_i,\bar{\lambda}_i,\bar{\beta}_i)$ if
$\EpecS \in \{\Epec^{\ind},\Epec^{\bl}\}$ or at 
$(x_i,y_i,\bar{\eta}_i,\bar{\mu}_i,\bar{\lambda}_i,\bar{\gamma}_i,\bar{\beta}_i)$ if $\EpecS =
\Epec^{\cc}$.
	\end{definition}

	Having defined the relevant objects, we now show that a
	strong-stationary point of $\PS$ is a Nash strong-stationary
	point of $\EpecS$ and a second-order strong-stationary point of
	$\PS$ is a second-order strong-stationary point of
	$\EpecS$. {
		For $i = 1, \hdots, N$, one may define a
		corresponding relaxed NLP associated with the $i\th$ leader's
		problem, namely $\LiS(x\mi,y\mi)$, by 
		employing the index sets $\Iscr_i$. These index sets are defined
		using $\tilde \Iscr$ and are given by\footnote{{The index sets associated with $\PS_{rnlp}$
			are denoted by $\tilde \Iscr$ while the index sets employed
				for specifying leader
				$i$'s relaxed NLP are denoted by $\Iscr_i$. Note that
					the cardinality of $\tilde \Iscr_i$ and $\Iscr_i$
					differs when considering the relaxed NLPs
					corresponding to $\Epec^{\ae}$ since every leader level problem
					contains equilibrium constraints of all the leaders.}} $\Iscr_i = \{\tilde \Iscr_1,
						\hdots, \tilde \Iscr_N\}.$

	\begin{proposition}\label{prop-stat} [Strong stationary points of $\PS$ and Nash
	strong stationary points of $\EpecS$]
Consider the multi-leader multi-follower game with shared
constraints $\EpecS$. Suppose $(x,y)$ is a strong-stationary point
	of $\PS$ and satisfies \eqref{ss-pae} with Lagrange multipliers $\left(\eta_i,\mu_i,\lambda_i,\beta_i\right)_{i=1}^N$.
	Then $(x,y)$ is a Nash strong-stationary point of $\EpecS$ and for
		$i = 1, \hdots, N$, $(x,y)$ satisfies \eqref{ss-Li} with
			Lagrange multipliers  defined as $\left(\bar \eta_i,\bar \mu_i,\bar \lambda_i,(\bar \beta_i^k)_{k=1}^N
							\right) = \left(\eta_i,\mu_i,\lambda_i,(\beta_k)_{k=1}^N
							\right)$
	Furthermore, if $(x,y)$ is a second-order strong stationary point of
		$\PS$ with multipliers $\left(\eta_i,\mu_i,\lambda_i,\beta_i\right)_{i=1}^N$, then $(x,y)$ is a
		Nash second-order strong stationary point of $\EpecS$ with
		firm $i$'s multipliers given by $\left(\bar \eta_i,\bar \mu_i,\bar \lambda_i,(\bar \beta_i^k)_{k=1}^N
									\right)$.
	\end{proposition}
	\begin{proof}
	Suppose $(x,y)$ is a strong stationary point $\PS$, \ie, suppose there exist
		multipliers $\eta,\mu, \lambda$ and $\beta$ such that for $(x,y)$, system \eqref{ss-pae} holds. 
	For each kind of $\EpecS$, we show that $(x,y)$ is a Nash strong
	stationary point of $\EpecS$. 
		One may then construct Lagrange multipliers to satisfy \eqref{ss-Li}. 
By comparison of \eqref{ss-pae} and \eqref{ss-Li}, we see that \eqref{ss-Li} 
admits a solution $(x,y)$ with multipliers $\bar{\eta}_i =\eta_i,\bar{\mu}_i = \mu_i,\bar{\lambda}_i=\lambda_i$ and $ \bar{\beta}_i^k = \beta_k$ for all $i,k$.  

	Now assume that $(x,y)$ is a second-order strong stationary point of
	$\PS$. To show that $(x,y)$ is a Nash second-order strong stationary
	point of $\EpecS$, we construct Lagrange multipliers as above. It is
	easy to see, that by construction, $\nabla_{x_i,y_i} \Lscr =
	\nabla_{x_i,y_i} \Lscr_i$ and 		$\nabla^2_{x_i,y_i} \Lscr_i =
	\nabla^2_{x_i,y_i} \Lscr$  for all $i \in \Nscr$, where $\Lscr_i$ is
	the Lagrangian of $\LiS$ evaluated at $(x,y)$ and the above
	constructed Lagrange multipliers. Furthermore, by comparing the
	feasible region of $\LiS$ with $\FscrS$, we observe that the active
	sets of $\LiS$ can be defined as $\Ascr_i(x,y) = 
		\left\{\tilde
		\Ascr_i^c(x,y), \tilde \Ascr_i^d(x,y), \tilde \Ascr_1^e(x,y),
		\hdots, \tilde \Ascr_N^e(x,y)\right\}.$ 

	Given the specification of the active set, we may now define a
		relaxed NLP corresponding to this active set as well as define
		the corresponding critical cone $S_i^*(x,y)$.
To prove the claim, we proceed by contradiction. If $(x,y)$ is not a
Nash second-order strong stationary point, then for some $i \in \{1,
	 \hdots, N\}$, the point $(x_i,y_i)$ does not satisfy second-order
	 strong stationary conditions, given $(x\mi,y\mi)$. Then there
	 exists a $w_i$ such that $w_i \in S_i^*(x,y)$ such that
	 $w_i\t\nabla^2_{x_i,y_i} {\Lscr^*_i} w_i \leq 0$.
We may now define ${\bf w} $ such that  $$ {\bf w} \triangleq (w_1,\hdots,w_N), $$ where
		$w_j = 0, j \neq i$. Since $w_i \in S_i^*(x,y)$, it follows that 
		$$ 0 = w_i\t \nabla_{x_i,y_i} \varphi_i(x,y) = w_i\t \nabla_{x_i,y_i}
		\pi(x,y).$$ 
		By definition of ${\bf w}$, it follows that ${\bf w}\t
		\nabla_{x,y} \pi(x,y) = 0.$ From the definition of ${\bf w}$ and
		by noting the constructions of $\Ascr_i(x,y)$, it can be seen
			that ${\bf w} \in S^*(x,y).$
%
As a	consequence, we have that 
		$$  0 \geq w_i^T \nabla_{x_i,y_i}^2 \Lscr_i w_i = {\bf w}^T
		\nabla_{x_i,y_i}^2 \Lscr  \ {\bf w}.$$   
		But this contradicts the hypothesis that $(x,y)$ is a second-order strong
		stationary point of $\PS$ and the result follows. 							
	\end{proof}

Notice that the form of equilibrium constraints was only used when considering stationarity concepts. The global equilibrium results did not use the explicit form of the equilibrium constraints and as such are applicable even for extensive form games.

As a final note, we recall it is not entirely necessary to employ the algebraic
characterization of the constraints in articulating strong stationarity
(cf.~\cite{flegel07optimality,henrion10note}). For instance, the authors examine the
optimality conditions of a disjunctive program  defined as 
\begin{align}
	\begin{aligned} 
		\min_{x} \quad & f(x) \\
		\st 	\quad & x \in \Lambda \triangleq \bigcup_{i=1}^m
		\Lambda_i, 
	\end{aligned}
\end{align}
where $\Lambda_i$ is a convex polyhedron for $i = 1, \hdots, m$. Such a
problem captures most MPEC models considered in the research literature.
The authors proceed to show that if the generalized Guignard constraint
qualification and a suitably defined intersection property holds at a
local minimizer $z$, then $z$ is a strong stationary point. Note that
the definition of strong stationarity relies on using the Fre\'{c}het
normal cone associated with $\Lambda_i$ rather than the algebraic
characterization of the sets. 

\subsection{Existence of global equilibria via fixed-point
	arguments}\label{sec:35}
The reaction map of the \mlmfg $ \Epec $ does not have the
properties required for applying fixed point theorems. However, the modified
formulation $ \Epec^{\ae} $, because of its shared constraint structure,
allows for the construction of a \textit{modified reaction map} whose
fixed points are equilibria of $\EpecS$ and has properties that are more
favorable for the application of fixed point theory. This leads to an
existence result for $\EpecS$ that uses fixed point theory and does
not assume the existence of a potential function.  We touch upon this
topic in this section. 

To define the modified reaction map let $\Psi :(X\times Y) \times (X \times Y)\rightarrow \Real$ be given by 
\begin{equation}
\Psi(x,y,\bar{x},\bar{y}) \triangleq \ds \sum_{i=1}^N \varphi_i(\bar{x}_i,\bar{y}_i;x^{-i}) \qquad \forall (x,y),(\bar{x},\bar{y}) \in X \times Y. \label{eq:psi}
\end{equation}
 and consider the {\em modified reaction map} 
 $\UpsilonS :X \times Y \rightarrow 2^{\FscrS}$, defined as 
\begin{equation}
 \UpsilonS(x,y) \triangleq \left\{ (\bar{x},\bar{y}) \in \FscrS  \ | \ \Psi(x,y,\bar{x},\bar{y}) = \inf_{(u,v) \in \FscrS} \Psi(x,y,u,v)\right\}. \label{eq:ups}
\end{equation}
We show below that a fixed point of $\UpsilonS$ is an equilibrium of
$\EpecS$. The map $\UpsilonS$ is analogous to that used by Rosen
\cite[Theorem 1]{rosen65existence}.  


\begin{theorem} [Fixed points of $\UpsilonS$ and equilibria of
$\EpecS$]\label{lem:rscr}
Consider the game \mlmfg $\EpecS$ with a feasible region mapping
$\OmegaS$. If $\UpsilonS$ admits a fixed point, the game $\EpecS$ admits an equilibrium.
\end{theorem}
\begin{proof}
Assume that the claim is false, \ie, assume there exists an  $(x,y) \in
\UpsilonS(x,y)$ such that for some $(u,v) \in \OmegaS(x,y)$ and an index $i \in \{1, \hdots,
	N\}$ we have 
$$\varphi_i(u_i,v_i;x^{-i}) < \varphi_i(x_i,y_i;x^{-i}). $$
Since $(u,v) \in \OmegaS(x,y)$, 
and since $\OmegaS$ satisfies \eqref{eq:shared}, 
we must have $(u_i,x^{-i},v_i,y^{-i}) \in
\Fscr$. But this means 
$$\Psi(x,y,(u_i,x^{-i}),(v_i,y^{-i})) <
\Psi(x,y,x,y),$$ a contradiction to $(x,y) \in \UpsilonS(x,y)$. 
\end{proof}
\noindent We further have that 	$\UpsilonS$ is upper semicontinuous under mild conditions.
\begin{lemma}Let $\Psi$ be continuous on $X \times Y$ and assume that $ X \times Y $ is compact. Then \label{lem:basic}
$\UpsilonS$ is upper semicontinuous. 
If $\UpsilonS$ is single-valued, then it is continuous (as a single-valued function).
\end{lemma}
\begin{proof}
By compactness of $ X \times Y $ the infimum in the definition of $
\UpsilonS $ is achieved. Upper semi-continuity of $ \UpsilonS $ follows
from classical stability results (see \eg, Hogan \cite{hogan73point}).
The last claim follows as a special case of upper semicontinuity of
set-valued maps for single-valued maps. 
\end{proof}
By using the Eilenberg-Montgomery 
	fixed point theorem~\cite{eilenberg46montgomery}, we obtain an existence result.
\begin{theorem} \label{thm:eilen}
Consider the \mlmfg $\EpecS$ where the objective functions $\varphi_i,
i\in \Nscr$ are continuous. Suppose $X\times Y$ is nonempty, compact and
convex. Suppose if $\UpsilonS$ satisfies one of the following:
\begin{enumerate}
\item[(i)] Single-valued on $X\times Y$; 
\item[(ii)] Multi-valued on $X\times Y$ with contractible images.
\end{enumerate}
Then $\UpsilonS$ admits a fixed point and  $\EpecS$ admits an equilibrium.
\end{theorem}
\begin{proof}
$\UpsilonS$ may be taken to be a mapping from the compact convex set $X \times Y$ to subsets of
$X \times Y$. If $\UpsilonS$ is single-valued, Lemma \ref{lem:basic}
implies that it is continuous. Consequently, by Brouwer's fixed point
theorem there exists a fixed point of $\Upsilon$. If $\UpsilonS$ is
multi-valued, then Lemma \ref{lem:basic} shows that $\UpsilonS$ is
upper semicontinous. Then since $\UpsilonS$ is contractible-valued, by the
Eilenberg-Montgomery fixed point theorem~\cite[Theorem
1]{eilenberg46montgomery}, there exists a fixed point of $\UpsilonS$. In
each of these cases, since there exists a fixed point of $\UpsilonS$,
   from Theorem \ref{lem:rscr}, $\EpecS$ admits an equilibrium.
\end{proof}

A natural question is when such conditions are useful. In general,
convexity or contractibility  of images of $\UpsilonS$ is not immediate;
however, if there are specific settings where such claims can be made,
then the aforementioned results are powerful in that they do not require
leader payoffs to admit potential functions.  It is not true that every
equilibrium of the \mlmfg with shared constraints is a fixed point of
$\UpsilonS$; existence of a fixed point to $\UpsilonS$ is only a
sufficient condition for such an equilibrium to exist.  This can be
checked easily by considering a hypothetical case with convex $\FscrS$,
wherein it is well known that fixed points of the reaction map and the
modified reaction map can be very different.
In~\cite{kulkarni09refinement}, Kulkarni and Shanbhag discuss these
issues in detail for convex \scgs; in general, there exist equilibria
that are not fixed points of $\UpsilonS$ and also games for which there
are equilibria, but no fixed points to $\UpsilonS$.

\begin{remarkc} {\bf (Relationship to variational equilibria in
			convex shared constraint games)}
When $\FscrS$ is convex, and $ \varphi_i(x_i;x\mi) $ is convex in $ x_i
$ for all $ x\mi $ and all $ i $, Lemma \ref{lem:rscr} and the map
$\UpsilonS$ also
has an interesting connection with the ``variational equilibrium''
\cite{facchinei07generalized,pang09convex,kulkarni09refinement} in games
with convex shared constraints. 	In this setting, these games are
typically referred to as generalized Nash games and equilibria of such
games are referred to as generalized Nash equilibria (GNE).  The
variational equilibrium is defined as the solution of the variational
inequality VI$(\FscrS,F)$, where $F = (\nabla_1 \varphi_1, \hdots,
		\nabla_N \varphi_N)$. By convexity, it can be easily seen that $
\VI(\FscrS,F) $ equals the set of fixed points of $\UpsilonS$.  The
variational equilibrium is the  generalized Nash equilibrium at which
the Lagrange multipliers, corresponding to the shared constraints, are
identical across players.  These multipliers can be interpreted as
the shadow prices of the associated constraints.  Furthermore, when
these prices are equal, the equilibria can be viewed as corresponding to
a uniform auction price while disparities in prices are a consequence of
discriminatory prices. The above observations form the basis of a
detailed study  of the VE and the GNE~\cite{kulkarni09refinement} where
we show that under general conditions, if a GNE exists, a VE also
exists, in which case the VE is said to be a \textit{refinement} of the
GNE~\cite{selten75reexamination,basar99dynamic}.  Furthermore, for
potential games with potential function $ \pi $, $ F\equiv \nabla \pi $,
		  whereby $ \VI(\FscrS,F) $ is equivalent to $ \VI(\FscrS,\nabla
				  \pi).$ Thus in a potential game with shared
		  constraints, every VE is also a stationary point of the
		  potential function over the shared constraint.  Coming back to
		  the game $ \Epec^{\ae} $, a stationary point of $ \PS $ is
		  therefore akin to a something like a VE of this formulation. 
\end{remarkc}

Theorem \ref{thm:eilen} did not invoke the existence of a potential
function. Nonetheless, there is a close relation between the minimizer
of the potential function, \ie the solution of problem $\PS$,  and the
fixed points of $\UpsilonS$. We formalize it through the following
definition.
\begin{definition} \label{def:lfp}
A {\em stationary fixed point} of $\UpsilonS$ is a point $(x,y) \in \FscrS$ with the property that $(x,y)$ is satisfies the stationarity conditions of the minimization of $\Psi(x,y,u,v)$ over $(u,v) \in \Fscr$ \ie,
\[\bar{\nabla}\Psi(x,y,x,y)\t d \geq 0 \quad \quad \forall\ d \in \Tscr((x,y);\FscrS), \]
where $\left.\bar{\nabla}\Psi(x,y,x,y) \triangleq \frac{\partial }{\partial (u,v)}\Psi(x,y,u,v)\right\rvert_{(u,v) = (x,y)}$. 
\end{definition}
If $ \Epec^{\ae} $ is a potential game with potential function $ \pi $, then we have $\nabla \pi(x,y) \equiv    \bar{\nabla}\Psi(x,y,x,y).$ Consequently, we have the following relation. 
\begin{proposition} \label{thm:lfp}
Let $\EpecS$ be a potential game with potential function $\pi$.
$ (x,y)\in \FscrS $ is a stationary point of the minimization $\pi$ over $\FscrS$ if and only if $ (x,y) $ is a stationary fixed point
of $\UpsilonS$. 
\end{proposition}

\section{Recovery of equilibria of $\Epec$}\label{sec:4}
The prior section has concentrated on the development of existence
	statements for the equilibria associated with the  shared-constraint
		(modified) equilibrium problem, denoted by $\EpecS$. While it
		has been shown that the equilibria of $\Epec$ are indeed
		equilibria of $\EpecS$, it remains unclear as to how one may
		obtain the equilibria of $\Epec$.  In Section~\ref{sec:41}, we
		consider settings where the Nash-stationary equilibria of
		$\Epec$ may indeed be recovered. A more refined statement is
		provided in Section~\ref{sec:42} under the assumption that
		follower equilibria are unique as a function of leader-level
		decisions.

\subsection{Obtaining Nash stationary points of $\Epec$}\label{sec:41}
In this section, we begin by providing an intuition about the
relationship between the equilibria of the original game and its shared
constraint modification by considering a convex generalized Nash game.
Consider a Nash game in which player $i$ has strategies $x_i,y_i$, objective $f_i(x_i,y_i;x\mi)$ and 
a nonlinear constraint $h(x,y_i) \geq 0$, where for any $x\mi$,
  $h(x_i,y_i;x\mi)$ 
and $f_i(x_i,y_i;x\mi)$ are concave and convex in $x_i,y_i$,
	respectively.  Specifically, player $i$ solves $A_i(x\mi)$, defined next.
$$	\problemsmall{A$_i(x\mi)$}
	{x_i,y_i}
	{f_i(x_i,y_i;x\mi) }
				 {\begin{array}{r@{\ }c@{\ }l}
		h(x,y_i)  \geq 0. \qquad (\lambda_1)
\end{array}} $$
We refer to this game as $G$ and corresponds to $\Epec$. Suppose A$_i(x\mi)$ is a convex optimization problem for each $x\mi$. The shared constraint modification of this problem akin to the ``ae'' formulation is given by the
following:
$$	\problemsmall{A$^{\ae}_i(x\mi,y\mi)$}
	{x_i,y_i}
	{f_i(x_i,y_i;x\mi) }
				 {\begin{array}{r@{\ }c@{\ }l}
		h(x,y_1)  & \geq & 0 \ \quad (\lambda_{11}) \\
				& \vdots & \\
		h(x,y_N)  & \geq & 0. \quad (\lambda_{1N})
\end{array}} $$
Notice that since $h(\cdot,\cdot)$ is concave in $x_i,y_i$ for all
$x\mi$, A$^{\ae}_i(x\mi,y\mi)$ is a convex optimization problem. Denote this game as $G^{\ae}.$ 
Our first result relates equilibria of $G$ to that of $G^{\ae}$:
\begin{lemma} The point $(x,y)$ is an equilibrium
of $G$ with multipliers $\lambda_1, \hdots, \lambda_N$ if and only if  $(x,y)$ is an equilibrium of $G^{\ae}$ with
\begin{align} \lambda_{ii}  = \lambda_i \mbox{ and } \lambda_{ij}  = 0, \forall j
\neq i.\label{multi-ae}\end{align} 
\end{lemma}
\begin{proof}
By the same logic as in Prop,~\ref{eq-orig-ae} (iii), any equilibrium of $G$
	is an equilibrium $G^{\ae}$. Since problems $A_i,i \in \Nscr,$ and $A_i^{\ae}, i\in \Nscr,$ are convex optimization problems, the aggregated 
	KKT conditions of individual players are necessary and sufficient for $(x,y)$ to be an equilibrium of $G$ and $G^{\ae},$ respectively. 
\eqref{multi-ae} now 
	follows from a examination of the KKT conditions.
\end{proof}

A more general
statement is available in the context of $\Epec$ and $\EpecS$ by
examining the strong stationarity conditions (the proof is straightforward; we skip it).

\begin{proposition} [Strong stationary points of  $\EpecS$ and $\Epec$]
Consider the multi-leader multi-follower games $\Epec$ and $\EpecS$.
Then the following hold:
\begin{enumerate}
\item $(x,y)$ is a Nash strong-stationary point of $\EpecS$ satisfying \eqref{ss-Li} with Lagrange
multipliers  $\left(\bar \eta_i,\bar \mu_i,\bar
\lambda_i,(\bar \beta_i^k)_{k=1}^N \right)$ and   $(\bar \beta^i_k)
	= 0, \forall k \neq i, \forall i$ if and only if $(x,y)$ is a Nash
	strong-stationary point of $\Epec$ with multipliers $\left(\bar \eta_i,\bar \mu_i,\bar
\lambda_i,\bar \beta_i^i \right)$.  
	\item Furthermore, $(x,y)$ is a Nash second-order strong-stationary point of
	$\EpecS$ with firm $i$'s multipliers given by $\left(\bar \eta_i,\bar \mu_i,\bar \lambda_i,(\bar \beta_i^k)_{k=1}^N
									\right)$ and  $(\bar \beta^i_k)
	= 0, \forall k \neq i, \forall i$ if and only if  $(x,y)$ is a Nash
	second-order strong-stationary point of $\Epec$ with multipliers $\left(\bar \eta_i,\bar \mu_i,\bar
\lambda_i,\bar \beta_i^i \right)$. 
	\end{enumerate}
	\end{proposition}

 In effect, one can inspect the multipliers of an equilibrium of the
shared-constraint game to ascertain whether indeed such a point is an
equilibrium of the original equilibrium problem. Yet, much more can be
drawn from this observation. Through the shred-constraint modification,
a generalized Nash game is converted to a shared-constraint
(generalized) Nash game. Furthermore, this modified  game admits
a set of generalized Nash equilibria (GNE) that contain two
important sets of equilibria (each of which may be empty):\\

\noindent {\bf Equilibria of $\Epec$:} These equilibria correspond to
points characterized by multipliers that display a precise form as
specified by Prop.~\ref{prop-stat}.\\ 

\noindent {\bf Variational equilibria of $\EpecS$:} These equilibria are
defined by a common Lagrange multiplier across every agent. Such
equilibria have proved to be particularly relevant in the context of
convex shared-constraint Nash games where the {\em common} Lagrange
multiplier is seen as the {\em uniform auction price}~\cite{kulkarni09refinement}. Moreover, in the
context of convex Nash games, VE may be obtained through the  solution
of a  variational inequality problem~\cite{kulkarni09refinement,kulkarni12revisiting}. In
the current context, under an assumption of potentiality on the
leader-level problems, such equilibria may be derived by the solution to
a suitably defined  optimization problem, such as P$^{\rm ae}$. \\

We believe that viewing the equilibria of $\Epec$ as particular equilibria of
$\Epec^{\ae}$ allows for two important directions:
\begin{enumerate}
\item {\em Pathways to existence statements:} The shared constraint
formulation admits a larger set of equilibria, that includes equilibria
of the conventional formulation (if they exist). We have seen that under mild assumptions, existence
of equilibrium akin to variational equilibria can be guaranteed. This may be
a stepping stone towards developing an approach for
claiming existence of equilibria of $\Epec$.
\item {\em Tools for equilibrium computation and selection:} Equilibrium
computation is a crucial concern in the design of markets, a realm where
such problems routinely arise. Yet, such designs are plagued by a key
challenge in that equilibria are not readily computable. If the objectives admit a potential function, this formulation provides two crucial
benefits. First, it allows for computing global variational equilibria
through the solution of a single optimization problem. Second,  if one
takes the view that the conventional formulation is the ``correct''
formulation, the modification may provide a means to arriving at an equilibrium of the conventional formulation, provided it exists. 
\end{enumerate}

\subsection{Recovery of equilibria when follower equilibria are
	unique}
\label{sec:42}
%
\def\S{^{\rm ae}}
Assume that for every $x \in X$, $\Sscr(x)$ is a singleton belonging to $\cap_{i\in \Nscr} Y_i$ in which case we have $\Fscr = \{ (x,y) \ | \ x \in X, y = \Sscr^N(x)\} = \Fscr^{\ae}$. 
Substituting the follower equilibrium tuple $y$ in terms of $x$ in the
definition of $\Upsilon\us{^{\ae}}$, we define 
$\Gamma\S : X \rightarrow 2^X$, as follows:
$$ \Gamma\S(x) \triangleq \arg \min_{u \in X} \Psi(x,\Sscr^N(x),u,\Sscr^N(u)).$$
If $x$ is a fixed point of $\Gamma\S$, then $(x,\Sscr^N(x))$ an equilibrium of $\Epec^{\ae}$. 
  For simplicity of exposition we will refer to fixed points of $\Gamma\S$ as ``equilibria'' of $\Epec^{\ae}$.  Now consider the original formulation $\Epec$ and rewrite the leader problem L$_i$ in the following form.
$$
	\problem{L$_i(x^{-i})$}
	{x_i}
	{\varphi_i(x_i,\Sscr(x);x^{-i})}
				 {\begin{array}{r@{\ }c@{\ }l}
		x_i &\in& X_i, \\
		\end{array}}
	$$
It is easy to see that an ``equilibrium'' of this game is the same as a fixed point of $\Gamma : X \rightarrow 2^{X}$, where 
$$ \Gamma(x) \triangleq \arg \min_{u\in X} \sum_{i=1}^N \varphi_i(u_i,\Sscr(u_i,x^{-i});x^{-i}) = \arg \min_{u\in X} 
\Psi(x,\Sscr^N(x),u,\Sscr(u_1,x^{-1}),\hdots,\Sscr(u_N,x^{-N})). $$
This follows from noting that $ X $ is a Cartesian product of $ X_1,\hdots,X_N. $ 
The next theorem exploits the similarity between $\Gamma\S$ and $\Gamma$ 
to develop conditions under which the fixed points of $\Gamma\S$ are also fixed points of $\Gamma$. 
\begin{theorem}
Suppose for all $x \in X$, $\Sscr(x)$ is a singleton lying in $\cap_{i\in \Nscr} Y_i$ and let the objectives of players be such 
that 
$$ \Psi(x,\Sscr^N(x),u,\Sscr^N(u)) \leq  \Psi(x,\Sscr^N(x),u,\Sscr(u_1,x^{-1}),\hdots,\Sscr(u_N,x^{-N})), \qquad \forall \ u, x \in X .$$
Then every fixed point of $\Gamma\S$ is also a fixed point of $\Gamma$ and thus an equilibrium of $\Epec$. In particular, if $\Gamma\S$ admits a 
fixed point, the conventional formulation $\Epec$ admits an equilibrium.
\end{theorem}
\begin{proof}
If $x$ is a fixed point of $\Gamma\S$, 
$$ \Psi(x,\Sscr^N(x),x,\Sscr^N(x)) \leq \Psi(x,\Sscr^N(x),u,\Sscr^N(u)) \quad \forall u \in X.$$
By the hypothesis of the theorem, we have
$$ \Psi(x,\Sscr^N(x),x,\Sscr^N(x)) \leq \Psi(x,\Sscr^N(x),u,\Sscr(u_1,x^{-1}),\hdots,\Sscr(u_N,x^{-N})),$$
which means $x$ is a fixed point of $\Gamma$.
\end{proof}

\begin{remarkc}
Notice the difference between $\Gamma$ and $\Gamma\S$. Importantly, observe that a fixed point of  
one is not necessarily a fixed point of the other. This may come as a surprise, considering that Proposition \ref{lem:ae} shows that equilibria of $\Epec$  are also equilibria of $\Epec^{\ae}$. But this ``contradiction'' 
can be explained by noticing that a fixed point of $\Gamma$ {\em is} an equilibrium of 
$\Epec^{\ae}$, but such an equilibrium need not be a fixed point of $\Gamma\S$. Since the fixed point formulation 
through $\Upsilon^{\ae}$ or $\Gamma\S$ is only a sufficient
condition for the existence of equilibria of $\Epec^{\ae}$, there may
exist equilibria of these games that are not necessarily fixed points of
the $\Gamma\S$. 
\end{remarkc}

\section{An example: a hierarchical Cournot game}\label{sec:cour}

In this section, we present  a \mlmfg from~\cite{sherali84multiple}
which when formulated in the conventional form  
has {an} equilibrium. Through  this game, {we will demonstrate the validity 
of } Propositions \ref{lem:ae} and Theorem \ref{thm:pot}. 

Below, we modify this game
	in the form of $\Epec^{\ae}$ and show that the claim made in Proposition \ref{lem:ae}
holds: the equilibrium of this game is also an equilibrium of its
modification. The example shows that equilibrium conditions of $\Epec^{\ae}$ have more variables than the conditions of $ \Epec $, and
thus {allow for more} ``degrees of freedom''  for their satisfaction. Cournot games, as noted in Section \ref{sec:examples}, admit potential functions. 
We then calculate the minimizer of the potential function of this game (\ie, the solution of $\PS$) and show that it is an equilibrium of the modified game, thereby verifying 
Theorem \ref{thm:pot}.
\begin{examplec} \label{ex:sherali1}
Let $\Epec$ be a game with $N$ \textit{identical} leaders and $n$ \textit{identical} followers. 
The follower strategies conjectured by leader $i$ are denoted by $\{y^f_i\}_{f=1,\hdots,n}$ (we use $f$ to index followers) and 
we let $\overline{y}_i^{-f}$ denote $\sum_{j \neq f} y^j_i$.
Leader $i$ solves {the following parametrized problem:} 
$$	\problem{L$_i(x\mi,y\mi)$}
	{x_i,y_i}
	{\half c x_i^2 - x_i\left(a -b(x_i+\sum_{j \neq i}x_j+ \sum_{f=1}^n
			y^f_i)\right)   }
				 {\begin{array}{r@{\ }c@{\ }l}
	y^f_i & = & \mbox{SOL(F}(\bar{y}^{-f}_i,x_i,x\mi)),\quad \forall\ f,\\
	x_i &\geq& 0,
\end{array}} $$   
where $y_i^f \in \Real$ is the conjecture of leader $i$ of the equilibrium strategy of follower $f$. 
Follower $f$ solves the problem (F$(\bar{y}^{-f},x)$):
$$ \problem{F$(\bar{y}^{-f},x)$}
	{y^f}
	{\half c (y^f)^2 - y^f\left( a - b (y^f  +  \sum_{j\neq f} y^j + \sum_{i\in \Nscr} x_i) \right) }
				 {\begin{array}{r@{\ }c@{\ }l}
y_f & \geq & 0,	
\end{array}}$$
where constants $a,b,c$ are positive real numbers. Since these constants are the same for all 
followers, equilibrium strategies of all followers are equal.
Consequently the follower equilibrium tuple conjectured by leader $i$ is given by $y_i =(\widehat{y}_i,\hdots,\widehat{y}_i)$, where $\widehat{y}_i$ satisfies $\widehat{y}_i \in \SOL({\rm F}((n-1)\widehat{y}_i,x))$ (since $\bar{y}^{-f} = (n-1)\widehat{y}_i$). For any $x$, there is a unique $\widehat{y}_i$ that satisfies this relation, given by 
\begin{equation}
 \widehat{y}_i = \begin{cases}
        (a-b\sum_j x_j)/(c+b(n+1)) & \mbox{if} \ \ 0 \leq \sum_j x_j \leq a/b,\\
0 & \mbox{if}\ \ \sum_j x_j > a/b.
       \end{cases} \label{eq:yhat}
\end{equation}  
By considering only the first of above cases in \eqref{eq:yhat}, we get a restricted game where leader $i$ solves
$$	\problem{L$'_i(x\mi,\widehat{y}\mi)$}
	{x_i,\widehat{y}_i}
	{\half c x_i^2 - x_i\left[a -b\left(x_i+\sum_{j \neq i}x_j+ n\widehat{y}_i\right)\right]    
	 }
				 {\begin{array}{r@{\ }c@{\ }ll}
				 \widehat{y}_i & = & \frac{a-b\sum_{j}x_j}{c+b(n+1)},\quad \quad &:\lambdabar_i\\
				 \sum_j x_j    & \leq  & a/b, &: \mubar_i \\
	x_i &\geq& 0.
\end{array}} $$   
This is a generalized Nash game {with coupled but not} shared
constraints. However, since the optimization problems of the leaders are
convex (this is not obvious; see \cite[Lemma 1]{sherali84multiple} for a
		proof), we may use the {first-order} KKT conditions to derive
an equilibrium. Let $\lambdabar_i$ be the Lagrange multiplier
corresponding the constraint ``$\widehat{y}_i =
\frac{a-b\sum_{j}x_j}{c+b(n+1)}$''. The equilibrium conditions of this
game are 
\begin{align}\label{eq:exshared1}
\left\{
\begin{aligned}
 0 \leq x_i &\perp \left( c+b \right)x_i  -a + b \left(\sum_{j}x_j +n\widehat{y}_i\right) + \frac{b}{c+b(n+1)}\lambdabar_i + \mubar_i \geq 0, \\
 \widehat{y}_i & =  \frac{a-b\sum_{j}x_j}{c+b(n+1)},  \\
 0 \leq \mubar_i &\perp a/b  - \sum_j x_j \geq 0, \\
             0 &= nbx_i + \lambdabar_i. \quad \quad \quad \quad \quad  
\end{aligned} \right\} \forall\ i \in \Nscr.
\end{align}
We can verify that the tuple $x=x^*$ where all leaders play the same strategy $\widehat{x}$, \ie 
$x^*_i = \widehat{x}$ for all $i$ with  $\widehat{x}$ given by 
$$\ds \widehat{x}= \frac{a(b+c)}{b(b+c)(N+1) +c(b+c) + bcn},$$ 
satisfies equilibrium conditions for the restricted game $\{{\rm L}_i'\}_{i\in \Nscr}$. 
The optimal Lagrange multiplier is given by $\lambdabar_i^* = -nbx^*_i, \mubar^*_i = 0$. 
It can then be verified that this equilibrium also satisfies the requirement $\sum_i x^*_i < a/b$, 
whereby it is an equilibrium of the original game. The other case of $\widehat{y}_i=0$ does not result 
in an equilibrium that satisfies $\sum_i x_i >a/b$, and consequently $x^*$ is the only equilibrium.

\paragraph{Verifying Proposition \ref{lem:ae} (An equilibria of $\Epec^{}$
		is an equilibrium of $\Epec^{\ae}$):} Let us now consider this game modified as $\Epec^{\ae}$. 
$$	\problem{L$^{\ae}_i(x\mi,y\mi)$}
	{x_i,y_i}
	{\half c x_i^2 - x_i(a -b(x_i+\sum_{j \neq i}x_j+ \sum_{f=1}^n y^f_i))   }
				 {\begin{array}{r@{\ }c@{\ }l}
	y^f_k & = & \mbox{SOL(F}(\bar{y}^{-f}_k,x_k,x^{-k})),\quad \forall\ f, \forall\ k=1,\hdots,N\\
	 x_i &\geq& 0.
\end{array}} $$   
Notice that the equilibrium constraint is now for all $f$ and for all $k$. For any $k$, 
the equilibrium constraint may be simplified using \eqref{eq:yhat}, giving an equation in $\widehat{y}_k$. It is easy to check that this game also admits no equilibrium with $\sum_j x_j >a/b$. Thus, this game is equivalent to the game where $\sum_j x_j $ is constrained to be in $[0,a/b]$. For such values of $\sum_j x_j$, the first case of \eqref{eq:yhat} applies, and it gives us a game where leader $i$ solves
$$	\problem{L$^{\ae}_i(x\mi,\widehat{y}\mi)$}
	{x_i,y_i}
	{\half c x_i^2 - x_i\left[a -b\left(x_i+\sum_{j \neq i}x_j+ n\widehat{y}_i\right)\right]    }
				 {\begin{array}{r@{\ }c@{\ }ll}
	\widehat{y}_k &= &         \frac{a-b\sum_j x_j}{c+b(n+1)}, \quad &:\lambda_{i}^k,\quad k=1,\hdots,N \\
	\sum_j x_j & \leq & a/b, \quad &: \mu_i,\\
	 x_i &\geq& 0.
\end{array}} $$   
This is a generalized Nash game with (convex) shared constraints and convex optimization problems for leaders. Let $\lambda_{i}^k$ be the 
Lagrange multiplier corresponding to the constraint ``$\widehat{y}_k =         \frac{a-b\sum_j x_j}{c+b(n+1)}$'' 
in the problem L$_i$. The equilibrium conditions for the 
generalized Nash equilibrium (see \cite{kulkarni09refinement}) of this game are
\begin{align} \label{eq:exshared}
\left\{
\begin{aligned}
 0 \leq x_i &\perp (b+c)x_i   -a + b \left(  \sum_{j} x_j + n\widehat{y}_i\right) + \frac{b}{c+b(n+1)}\sum_{j =1}^N \lambda^j_i + \mu_i\geq 0,\\
 \widehat{y}_i &=          \frac{a-b\sum_j x_j}{c+b(n+1)}, \\
 0 \leq \mu_i &\perp a/b -\sum_j x_j    \geq 0, \\
        0 &=    nbx_i + \lambda^i_i,  \quad \quad \quad \quad \quad  
	\end{aligned}\right\}\forall\ i \in \Nscr.
\end{align}
Notice that the Lagrange multipliers $\lambda_i^j$ for $j\neq i$ are
unconstrained barring their presence in the first condition of
\eqref{eq:exshared}.  Comparing \eqref{eq:exshared} and
\eqref{eq:exshared1},  we see that if $\lambdabar^*,x^*$
solve system \eqref{eq:exshared1}, then $x=x^*$ and $\lambda_i^j =
\lambdabar^*_i\I{j=i}$ for all $i,j\in \Nscr$ gives a solution to  system
\eqref{eq:exshared}. Consequently, an equilibrium of the original game
$\Epec$ is an equilibrium of $\Epec^{\ae}$. We have thereby verified
Proposition \ref{lem:ae} for this problem.

\gap

The presence of surplus Lagrange multipliers provides us with more
variables than the number of equations, whereby existence of solutions
is easier to guarantee. An equilibrium of $\Epec$ is an equilibrium of
the modified game $\Epec^{\ae}$ with a specific configuration of the
vector of Lagrange multipliers. Consequently, if an equilibrium exists
to $\Epec^{\ae}$, there is no guarantee that there exists one to the
original game $\Epec$.

\paragraph{Verifying Theorem \ref{thm:pot} (Global minimizer of $\pi$ is
		an equilibrium of $\Epec^{\ae}$):} Applying the same arguments
as before, we can effectively consider the strategies of leader $i$ in game
$\Epec^{\ae}$ as $x_i$ and $\widehat{y}_i$. Further, suppose the
function $\pi$ is given by 
\[\pi(x,\widehat{y}) = \half c \sum_i x_i^2 - a \sum_i x_i +b\left( \sum_i x_i^2  + \sum_{i<j}x_ix_j\right)  + nb \sum_i x_i \widehat{y}_i,  \]
where $\widehat{y} \triangleq (\widehat{y}_1,\hdots,\widehat{y}_N)$. 
Notice that the map $F$ (cf., Lemma \ref{lem:jac}) is given by
\[F(x,\widehat{y}) =\ds \pmat{\frac{\partial \varphi_i}{\partial x_i}\\ \frac{\partial \varphi_i}{\partial \widehat{y}_i}}_{i \in \Nscr}=\pmat{(b+c)x_i   -a + b \left(  \sum_{j} x_j + n\widehat{y}_i\right) \\ 
nbx_i}_{i\in \Nscr},    \]
and that $\nabla \pi \equiv F$, whereby $\pi$ is a potential function for $\Epec^{\ae}$. The set $\Fscr^{\ae}$ for this game is 
\[\Fscr^{\ae} = \{(x,\widehat{y})\ |\ x \geq 0,\ \widehat{y}_i {\rm\ satisfies \ \eqref{eq:yhat}} \ \forall\  i\}.    \]
We now determine the global minimizer of $\pi$ over $\Fscr^{\ae}$. A
{significant} difficulty in characterizing the global minimizer of $\pi$ is
that $\pi$ is not necessarily convex (despite the convexity of the
		objectives of leaders in their own variables). 

\gap

We argue as follows. By membership of $(x,\widehat{y})$ in $\Fscr^{\ae}$, 
we either have 
$\widehat{y}_i =  \frac{a-b\sum_j x_j}{c+b(n+1)}$ for all $i$ or we have $\widehat{y}_i = 0$ for all $i$. 
Substituting for $\widehat{y}$, we can write $\pi$ as a function only of $x$ (with a slight abuse of notation)
\[\pi(x) =  \begin{cases}
\half c \sum_i x_i^2 - a \sum_i x_i +b\left( \sum_i x_i^2  + \sum_{i<j}x_ix_j\right)     + nb\frac{a-b\sum_i x_i}{c+b(n+1)} \sum_ix_i & \eef 0 \leq \sum_i x_i \leq a/b, \\
\half c \sum_i x_i^2 - a \sum_i x_i +b\left( \sum_i x_i^2  + \sum_{i<j}x_ix_j\right)   & \eef \sum_i x_i > a/b.
                            \end{cases}
\]
By symmetry, the values $x_i$ that minimize $\pi$ are equal for all $i$. Let $x_i = x'$ for all $i$ be the minimizer. 
Then, 
\[\pi(x') =  \begin{cases}
\half N c x'^2 - aNx' +b\left( N x'^2  + \frac{N(N-1)}{2}x'^2\right)     + nb\frac{a-bNx'}{c+b(n+1)} Nx' & \eef 0 \leq x'\leq a/(Nb), \\
\half N c x'^2 - a N x' +b\left( N x'^2  + \frac{N(N-1)}{2}x'^2\right)   & \eef x'> a/(Nb).
                            \end{cases}
\]
The right hand derivative of $\pi$ at $x'=a/b$ is positive, $\nabla \pi(x')^+|_{x'=a/b} = N[\frac{ac}{Nb} -a + a(N+1)]>0$. Furthermore $\pi$ is increasing and coercive for $x'>a/(Nb)$, and consequently the minimizer of $\pi$ lies in $[0,a/(Nb)]$. 
Since $x'$ is a global minimizer of $\pi$, $x'$ necessarily satisfies the first-order KKT conditions for the minimization of $\pi$ over $[0,a/(Nb)]$:
\begin{align}
 0 \leq x' &\perp N\left( (b+c)x' -a + bNx' + nb \frac{a-bNx'}{c+b(n+1)}- \frac{nb^2Nx'}{c+b(n+1)}\right) + \mu' \geq 0, \label{eq:exshared2}\\
 0 \leq \mu' &\perp a/(Nb) - x' \geq 0, \non
\end{align}
where $\mu'$ is the Lagrange multiplier for the constraint `$a/(Nb) -
x'\geq 0$'. If $x',\mu'$ is a solution of system \eqref{eq:exshared2},
	then $x_i=x'$, $\mu_i= \mu'$ and $\lambda_i^j =-nbx_j=-nbx'$ for all
	$i,j\in \Nscr$, solves system \eqref{eq:exshared} for the equilibrium
	of $\Epec^{\ae}$. Consequently $x=(x',\hdots,x')$ is an equilibrium
	of $\Epec^{\ae}$. This verifies Theorem \ref{thm:pot}.
	
	 {It should be emphasized 	that we have claimed that a solution to the concatenated
		first-order KKT conditions of the minimization of $\pi$ over $\Fscr^{\ae}$ is a global equilibrium of $\Epec^{\ae}$, a claim that
		is valid because the leader problems in $\Epec^{\ae}$ have been reduced to convex
		problems. In the case where $\Sscr$ is single-valued (as it was in this example), this is 
		possible because we could argue that for values $(x,y)$ of interest, the equation $y=\Sscr(x)$ is linear in $x,y$.} 
\end{examplec}

\section{Conclusions}\label{conc}
To summarize, while general existence results for the original formulation of \mlmfgs (more generally EPECs) are rare, we observed that a modified
formulation in which each player is constrained by the equilibrium
constraints of all players contains the equilibria of the original
game when this game does indeed admit equilibria.  This modified 
game admits  a shared constraint structure and when the leaders' objectives admit a potential function, the set of
global minimizers of the potential function over the shared constraint
are the equilibria of the modified \mlmfg. Similar statements were made relating the stationary points of such a problem to
the associated Nash B-stationary and strong-stationary equilibria.  In
effect, the above results reduced the question of the existence of an
equilibrium to that of the solvability of an optimization problem,  in
particular a mathematical program with equilibrium constraints.  This
solvability can be claimed under fairly standard conditions that are
tractable and verifiable -- \eg, coercive objective over a nonempty
feasible region -- and existence of a global equilibrium is seen to follow. 

It was further seen that the equilibria of the original
	formulation can be viewed as equilibria of the shared-constraint
		modification in which the associated Lagrange multipliers take
		on a specific form. This
		understanding may have much potential in deriving existence
		statements as well as computational schemes for the equilibria
		arising from  the original formulation.  
 We concluded with an
	application of our findings on a multi-leader multi-follower
	symmetric Cournot game.

\bibliographystyle{../Mylatexfiles/plainini}
\bibliography{../Mylatexfiles/ref,ref2}
\end{document}